\documentclass[11pt]{amsart}
\usepackage[english]{babel}
\usepackage{amsmath}
\usepackage{amsfonts}
\usepackage{amssymb}
\usepackage{amsthm}
\usepackage{color}
\usepackage{xargs}

\usepackage[english]{babel}
\usepackage{yfonts}
\usepackage[T1]{fontenc}
\usepackage[utf8x]{inputenc}
\usepackage{enumerate}
\usepackage{verbatim}
\usepackage{graphicx}
\usepackage{verbatim}
\usepackage{faktor}
\usepackage{xcolor}
\usepackage{xfrac}
\usepackage{tikz}
\usepackage[all]{xy}
\usepackage[mathscr]{euscript}

\renewcommand{\b}{\beta}

\newcommand{\calG}{\mathcal G}
\renewcommand{\L}{\Lambda}
\newcommand{\PSL}{{\rm PSL}}
\newcommand{\ga}{\gamma}

\newcommand{\deH}{\partial \mathbb H}
\renewcommand{\P}{{\rm P}}

\newcommand{\p}{{\rm p}}

\newcommand{\wt}{\widetilde}

\newcommand{\de}{\partial}

\newcommand{\calK}{\mathcal K}
\newcommand{\calX}{\mathcal X}

\newcommand{\bR}{\mathbf R}

\newcommand{\g}{\gamma}

\renewcommand{\l}{\lambda}

\newcommand{\G}{\Gamma}

\newcommand{\R}{\mathbb R}

\newcommand{\Z}{\mathbb Z}
\renewcommand{\P}{\mathbb P}
\renewcommand{\H}{\mathbb H}
\newcommand{\F}{\mathbb F}

\newcommand{\N}{\mathbb N}

\newcommand{\calC}{\mathcal C}
\newcommand{\SL}{{\rm SL}}
\renewcommand{\l}{\lambda}

\newcommand{\Cc}{\mathcal C}

\newcommand{\Ee}{\mathcal E}
\newcommand{\Gg}{\mathcal G}

\newcommand{\Xx}{\mathcal X}

\newcommand{\calL}{\mathcal L}

\newcommand{\Rr}{\mathcal R}
\newcommand{\Vv}{\mathcal V}
\newcommand{\MLc}{\mathcal{ML}_c}

\newcommand{\Sp}{{\rm Sp}}

\newcommand{\<}{\langle}
\renewcommand{\>}{\rangle}

\newcommand{\Syst}{\mathrm{Syst}}
\newcommandx{\ioo}[2]{I_{(#1,#2)}}
\newcommandx{\ioc}[2]{I_{(#1,#2]}}
\newcommandx{\ico}[2]{I_{[#1,#2)}}
\newcommandx{\icc}[2]{I_{[#1,#2]}}
\newcommand{\supp}{\mathrm{supp}}

\newcommand{\carrier}{\mathrm{carr}}
\newcommand{\CcK}{\mathcal C_{cc}^K}
\newcommand{\Ccc}{\mathcal C_{cc}}

\newcommand{\bq}{\begin{equation}}
\newcommand{\eq}{\end{equation}}
\newcommand{\bqn}{\begin{equation*}}
\newcommand{\eqn}{\end{equation*}}
\newcommand{\ba}{\begin{aligned}}
\newcommand{\ea}{\end{aligned}}
\newcommand{\be}{\begin{enumerate}}
\newcommand{\ee}{\end{enumerate}}
\newcommand{\bei}{\begin{itemize}}
\newcommand{\eei}{\end{itemize}}
\newcommand{\bsm}{\left(\begin{smallmatrix}}
\newcommand{\esm}{\end{smallmatrix}\right)}                   
\newcommand{\bpm}{\begin{pmatrix}}
\newcommand{\epm}{\end{pmatrix}}
\newcommand{\barr}{\begin{displaymath}\begin{array}{cccc}}
\newcommand{\earr}{\end{array}\end{displaymath}}
\newcommand{\barrl}{\begin{displaymath}\begin{array}{lcl}}
\newcommand{\earrl}{\end{array}\end{displaymath}}
\newcommand{\barl}{\begin{displaymath}\begin{array}{l}}
\newcommand{\earl}{\end{array}\end{displaymath}}
\newcommand{\bxym}{ \begin{displaymath}\xymatrix }
\newcommand{\exym}{\end{displaymath}}

\theoremstyle{plain}
\newtheorem{thm}{Theorem}[section]
\newtheorem{lem}[thm]{Lemma}
\newtheorem{prop}[thm]{Proposition}
\newtheorem{cor}[thm]{Corollary}
\newtheorem*{teo*}{Theorem}

\newtheorem{thm?}{Theorem}[section]

\theoremstyle{definition}
\newtheorem{example}[thm]{Example}

\newtheorem{defn}[thm]{Definition}
\newtheorem{remark}[thm]{Remark}

\newtheorem*{rem}{Remark}

\numberwithin{equation}{section}

\newcommand{\thismonth}{\ifcase\month 
  \or January\or February\or March\or April\or May\or June%
  \or July\or August\or September\or October\or November%
  \or December\fi}

\date{\today}
\begin{document}
\title{}

\title[A structure theorem for geodesic currents]%
      {A structure theorem for geodesic currents and length spectrum compactifications}
  
\author[]{M. Burger}
\address{Department Mathematik, ETH Zentrum, 
R\"amistrasse 101, CH-8092 Z\"urich, Switzerland}
\email{burger@math.ethz.ch}

\author[]{A. Iozzi}
\address{Department Mathematik, ETH Zentrum, 
R\"amistrasse 101, CH-8092 Z\"urich, Switzerland}
\email{iozzi@math.ethz.ch}

\author[]{A. Parreau}
\address{Univ. Grenoble Alpes, CNRS, Institut Fourier, F-38000
  Grenoble, France}
\email{Anne.Parreau@univ-grenoble-alpes.fr}

\author[]{M. B. Pozzetti}
\address{Mathematical Institute, Heidelberg University, Im Neuenheimer feld 205, 69120 Heidelberg, Germany }
\email{pozzetti@mathi.uni-heidelberg.de}
\thanks {Beatrice Pozzetti thanks Anna Wienhard for insightful conversations. Marc Burger thanks Francis Bonahon and Kasra Rafi for enlightening conversations on geodesic currents and Kasra Rafi for suggesting to prove a decomposition theorem for geodesic currents.
\\
\indent
Beatrice Pozzetti was partially supported by the SNF grant P2EZP2\_159117, and by the DFG project PO 2181/1-1. Marc Burger and Alessandra Iozzi were partially supported by the SNF grant 2-77196-16. Alessandra Iozzi acknowledges moreover support from U.S. National Science Foundation grants DMS 1107452, 1107263, 1107367 "RNMS: Geometric Structures and Representation Varieties" (the GEAR Network). 
Marc Burger thanks the Leverhulme Trust for supporting his visit to the University of Cambridge as Leverhulme Visiting Professor.  
Marc Burger, Alessandra Iozzi and Beatrice Pozzetti  would like to thank the Isaac Newton Institute for Mathematical Sciences, Cambridge, for support and hospitality during the program ``Non-Positive Curvature Group Actions and Cohomology'' where work on this paper was undertaken. This work was supported by EPSRC grant no P/K032208/1.}

\date{\today}
\begin{abstract}
We find a canonical decomposition of a geodesic current on a surface
of finite type arising from a topological decomposition of the surface
along special geodesics. We show that each component either is
associated to a measured lamination or has positive systole.
For a current with positive systole, we show that  the intersection
function on the set of closed curves is bilipschitz equivalent to the
length function with respect to a hyperbolic metric.
We show that the subset of currents with positive systole is open and
that the mapping class group acts properly discontinuously on it.
As an application, we obtain in the case of compact surfaces a structure theorem
on the length functions appearing in the length spectrum compactification both
of the Hitchin and of the maximal character varieties 
and determine therein an open set of discontinuity 
for the action of the mapping class group.
\end{abstract}

\maketitle


\section{Introduction}\label{sec:intro}
Let $\Sigma=\G\backslash \H^2$ be a complete hyperbolic surface of finite area. 
Recall that a geodesic current is a $\G$-invariant and flip invariant Radon measure 
on the space $(\partial\H^2)^{(2)}$ of distinct pairs of points in the boundary of $\H^2$. 
Geodesic currents occur in many different contexts. 
For instance geodesic currents play a fundamental role in the study of hyperbolic structures \cite{Bon}, 
negatively curved metrics \cite{Otal}, singular flat structures \cite{DLR}, Hitchin and maximal representations \cite{Martone_Zhang}. 

Our motivation comes form the study of various compactifications of maximal character varieties;  
we will present a simple application in this article.  
The aim of this paper is 
to exhibit a canonical decomposition of a general geodesic current 
 where either each component is associated to a measured lamination or 
 it has a behavior comparable to the one of a Liouville current.
In a forthcoming paper, continuing the study initiated in \cite{BP}
of maximally framed representations of surface groups into $\Sp(2n,\F)$,
where $\F$ is a non-Archimedean real closed field, we will show how to associate
to such a representation a geodesic current whose intersection function
on closed geodesics gives the length function.
Together with the results of this paper, 
this will lead to information concerning the real spectrum compactification
of maximal character varieties.


\medskip

Given a geodesic current $\mu$, 
we consider the set $\Ee_\mu$ of \emph{special} closed geodesics, namely
\bq\label{eq:E}
\ba
\Ee_\mu:=\{c\subset\Sigma:\,c\text{ is a closed geodesic such that }i(\mu, c)=0\text{ and}&\\
i(\mu,c')>0\text{ for every closed geodesic }c'\text{ with }i(c,c')>0&\}\,.
\ea
\eq
Here $i$ denotes the Bonahon intersection pairing on the set of geodesic currents,
which extends the geometric intersection number of closed geodesics 
and in fact the notation $i(\mu,c)$ really refers to the intersection number of $\mu$ and
the geodesic current $\delta_c$ associated to the closed geodesic $c$ (see Section \ref{sec:preliminaries} for more details).

Observe that $\Ee_\mu$ consists of pairwise disjoint simple closed curves that decompose the surface into a union of subsurfaces with geodesic boundary
\bq\label{eq:V}
\Sigma=\bigcup_{v\in\Vv_\mu} \Sigma_v\,.
\eq  
For such a subsurface $\Sigma_v$ let $\mathcal G_v\subset (\partial \H^2)^{(2)}$ be the set of geodesics whose projection lies in the interior $\mathring\Sigma_v$ of $\Sigma_v$, and, for a geodesic current $\nu$ let 
$$\Syst_{\Sigma_v}(\nu)=\inf\{i(\nu, c)|\,c\subset\mathring\Sigma_v \text{ closed geodesic }\}\,,$$
while $\Syst(\nu)$ refers to the analogous quantity with the infimum taken over all the geodesics in $\Sigma$.

\begin{thm}\label{thm_intro:dec}
Let $\mu$ be a geodesic current on a complete hyperbolic surface of finite area $\Sigma=\Gamma\backslash\H^2$,  
and let $\Ee_\mu$ and $\Vv_\mu$ as in \eqref{eq:E} and \eqref{eq:V}.
We have
$$\mu=\sum_{v\in \Vv} \mu_v+\sum_{c\in \Ee}\lambda_c\delta_{c}\,,$$
where $\mu_v$ is supported in $\mathcal G_v$ and $\delta_{c}$ is the geodesic current associated to the closed geodesic $c$. 
Furthermore, for every $v\in \Vv_\mu$ for which $\mu_v\neq 0$ precisely one of the following holds:
\begin{enumerate}
\item either $\Syst_{\Sigma_v}(\mu_v)>0$,
\item or $\mu_v$ is the geodesic current associated to a measured lamination 
 compactly supported in $\mathring\Sigma_v$ and intersecting every curve in $\Sigma_v$.
\end{enumerate}
\end{thm}

A similar result, in the specific case of currents $\mu$ associated to degenerations of singular flat structures on surfaces, was proven by Duchin, Leininger and Rafi \cite[Theorem 5]{DLR}. Note, however, that our definition of systole doesn't coincide with the one used in the proof of \cite[Theorem 5]{DLR}.

The second case of the dichotomy in Theorem \ref{thm_intro:dec} admits equivalent interesting characterizations, 
in particular using the new concept of $\mu$-somewhat short geodesics,
a generalization of the concept of closed geodesic with vanishing $\mu$-intersection: a geodesic $\sigma\subset \H^2$ is $\mu$-somewhat short if it doesn't intersect any geodesic in the support of $\mu$ (see Section \ref{subsec:1.0} for more details). We state here the new characterization in the simpler case in which the decomposition of Theorem \ref{thm_intro:dec} is trivial, and refer to Theorem \ref{thm:main_currents} for a more general statement.

\begin{thm}\label{thm_intro:main_currents}  Let $\mu$ be a geodesic current 
with positive intersection with every simple closed curve. 
Then the following are equivalent:
\be
\item\label{item:main_currentsI(1)} $\Syst(\mu)=0$;
\item\label{item:main_currentsI(2)} There exists a recurrent geodesic $(a,b)$ that is $\mu$-somewhat short;
\item\label{item:main_currentsI(3)} The support of $\mu$ is a $\Gamma$-invariant lamination in $\H^2$
whose projection on $\Sigma$ is compactly supported, minimal and surface filling.
\ee
\end{thm}
Recall that a lamination is surface filling if it intersects every simple closed curve. We say that a geodesic $(a,b)$ is recurrent 
if at least one among $a$ and $b$ does not correspond to a cusp.

The next result explains to which extent currents with positive systole behave like a Liouville current, that is a current whose intersection computes the length in a hyperbolic structure (see Example \ref{ex:2.1}(2)). 
To this end recall that the space of geodesic currents $\mathcal C(\Sigma)$ is endowed with the weak*-topology 
as topological dual of the space of continuous functions with compact support on $(\partial \H^2)^{(2)}$. 
The quotient $\P\mathcal C(\Sigma)$ of $\mathcal C(\Sigma)^*$, the space of non-zero currents, 
by the positive scalar multiplication is then compact (see \cite{B2}).

Choose $c_1,\ldots,c_n$ any collection of closed geodesics in $\Sigma$ 
such that $\bigcup_{i=1}^nc_i$ cuts $\Sigma$ into discs containing at most one cusp, 
and let us consider the geodesic current $\lambda=\sum_i\delta_{c_i}$. 
Then $\lambda$ is surface filling and we may represent points $[\mu]\in\P\mathcal C(\Sigma)$ 
by the corresponding representative $\mu$ with $i(\mu,\lambda)=1$.

\begin{thm}\label{thm_intro:posSyst}
The set $\Omega=\{[\mu]|\; \Syst(\mu)>0\}\subset\P\mathcal C(\Sigma)$ is open.
Moreover for every $[\mu]\in\Omega$, and for every $K\subset \Sigma$ compact, 
there exist a neighborhood $V_{[\mu]}$ of $\mu$ and constants $0<C_1\leq C_2$ 
such that for every $[\nu]\in V_{[\mu]}$ and every geodesic $c$ contained in $K$ we have 
$$C_1\ell(c)\leq i(\nu,c)\leq C_2\ell(c).$$
 Here $\ell(c)$ denotes the hyperbolic length of the geodesic $c$.
\end{thm}

This theorem is a consequence of a new characterization of currents with vanishing systole 
in term of their intersection with measured laminations (see Theorem \ref{thm:intersection} for a precise statement).

An immediate consequence of Theorem \ref{thm_intro:posSyst}  and the fact that the mapping class group acts properly discontinuously on the Teichm\"uller space is the following:
\begin{cor}
The mapping class group of $\Sigma$ acts properly discontinuously on $\Omega$.
\end{cor}

In a forthcoming paper we will use the results of this paper to show
that, for a general surface of finite type $\Sigma$, the mapping class
group of $\Sigma$ acts properly discontinuously on the character
variety of maximal representations.

We turn now to the applications to the length spectrum compactifications of maximal and Hitchin character varieties.
We assume from now on that $\Sigma=\Gamma\backslash\H^2$ is compact.
The extension of these results to the case in which $\Sigma$ is not compact requires additional tools 
and will be presented in a forthcoming paper.

Recall that a representation $\rho:\pi_1(\Sigma)\to\Sp(2n,\R)$ is maximal if its Toledo invariant achieves
its maximal value $n(2g-2)$; see \cite{BILW} for details. 
The space $\mathrm{Max}(\Sigma,n)$ of $\Sp(2n,\R)$-conjugacy classes of maximal representations
forms a union of connected components of the character variety of $\pi_1(\Sigma)$ in $\Sp(2n,\R)$.
The Hitchin component $\mathrm{Hit}(\Sigma,n)$ is the connected component of the character variety
of $\pi_1(\Sigma)$ in $\PSL(n,\R)$ containing $i_n\circ h$, where $i_n:\PSL(2,\R)\to\PSL(n,\R)$
is the $n$-dimensional irreducible representation of $\PSL(2,\R)$ and $h:\pi_1(\Sigma)\to\PSL(2,\R)$
is an orientation preserving hyperbolization.  
In the sequel  $\calX(\Sigma,n)$  will refer either to
$\mathrm{Max}(\Sigma,n)$ or to $\mathrm{Hit}(\Sigma,n)$.

In the spirit of Thurston and following Parreau \cite{Parreau12}, 
the space $\calX(\Sigma,n)$ 
can be compactified using length functions and in turn, length functions 
in the boundary come from certain isometric actions on affine buildings.
We now define these length functions in their proper context.
Let $\F$ be a real closed field containing $\R$, which in the case in which $\F\neq\R$,
we assume endowed with a non-Archimedean valuation.  
We let $\mathrm{G}_n(\F)$ be either $\Sp(2n,\F)$ or $\SL(n,\F)$,
and let $\Xx_n^\R$ be the associated symmetric space if $\F=\R$ 
or the associated Bruhat--Tits building if $\F$ is non-Archimdean.
Let $\nu:\mathrm{G}_n(\F)\to\overline{\mathfrak a^+}$ be the Weyl chamber valued translation vector
(see \cite[p.~3]{Parreau12}) and let $\|\,\,\|$ be the Weyl group invariant norm on $\mathfrak a$ defined as follows:
\begin{itemize}
\item in the case of the symplectic group,
\bqn
\overline{\mathfrak a^+}=\{(x_1,\dots,x_n)\in\R^n:\,x_1\geq\dots\geq x_n\geq0\}
\eqn
and
\bqn
\|(x_1,\dots,x_n)\|:=\sum_{i=1}^nx_i\,;
\eqn
\item in the case of the special linear group,
\bqn
\overline{\mathfrak a^+}=\{(x_1,\dots,x_n)\in\R^n:\,x_1\geq\dots\geq x_n\text{ and }\,x_1+\dots+ x_n=0\}
\eqn
and
\bqn
\|(x_1,\dots,x_n)\|:=x_1-x_n\,.
\eqn
\end{itemize}
Then for $g\in\mathrm{G}_n(\F)$ we define
\bqn
L(g)=\|\nu(g)\|\,.
\eqn

Given a representation $\rho$ of $\pi_1(\Sigma)$ into $\mathrm{G}_n(\F)$, 
if  $\mathscr{C}$ denotes  the set of conjugacy classes of non-trivial elements of $\pi_1(\Sigma)$, 
we define $L_\rho:\mathscr{C}\to\R$ by
\bqn
L_\rho(c):=L(\rho(\gamma))\,,
\eqn
where $\gamma\in\pi_1(\Sigma)$ represents $c\in\mathscr{C}$.  
Endowing  $\R_{\geq0}^\mathscr{C}$ with the product topology 
and $\P(\R_{\geq0}^\mathscr{C})$ with the quotient topology, we now
consider the continuous map $\calL:\calX(\Sigma,n) \to \P(\R_{\geq0}^\mathscr{C})$ 
induced by $\rho\mapsto[L_\rho]$.
Here and in the sequel, $[L]\in\P(\R_{\geq0}^\mathscr{C})$ denotes the equivalence class of a non-zero length function $L$
up to multiplication by positive scalars.
The boundary of $\calX(\Sigma,n)$ in the length compactification is
$$\partial\calL(\Sigma,n):=\cap_\calK \overline{\calL(\calX(\Sigma,n)-\calK)}$$
where the intersection is taken over all compact subsets $\calK\subset
\calX(\Sigma,n)$. It is a compact subset of $\P(\R_{\geq0}^\mathscr{C})$.
This compactification can be equivalently described as the closure of $\calX$ regarded 
as a subspace of the product of the Alexandrov compactification $\hat X$ and $\P(\R^\mathscr C_{\geq 0})$. 

Given $[L]\in\P(\R_{\geq0}^\mathscr{C})$, we consider in analogy with \eqref{eq:E} the set $\Ee_{[L]}$ of special geodesics
\bq\label{eq:EL}
\ba
\Ee_{[L]}:=\{c\subset\Sigma:\,c\text{ is a closed geodesic such that }L(c)=0\text{ and}&\\
L(c')>0\text{ for every closed geodesic }c'\text{ that intersects }c&\}\,.
\ea
\eq
Then $\Ee_{[L]}$ consists of pairwise disjoint simple closed geodesics that decompose the surface
into a union of subsurfaces with geodesic boundary
\bq\label{eq:VL}
\Sigma=\bigcup_{v\in\Vv_{[L]}} \Sigma_v\,.
\eq
Let 
\bqn
\mathrm{Syst}_{\Sigma_v}(L):=\inf\{L(c):c\subset\Sigma_v^\circ\text{ is closed}\}
\eqn
and
\bqn
\mathrm{Syst}(L):=\inf\{L(c):c\subset\Sigma\text{ is closed}\}\,.
\eqn

\begin{cor}\label{cor_intro:1.5}  Let $[L]$ be in $\partial\calL(\Sigma,n)$ and let $\Ee_{[L]}$
and $\Vv_{[L]}$ be as in \eqref{eq:EL} and \eqref{eq:VL}.  Then one of the following holds for $\Sigma_v$:
\be
\item $L$ vanishes on $\pi_1(\Sigma_v)$;
\item $L(c)>0$ for every closed geodesic $c\subset\Sigma_v^\circ$ and $L$ is the length function
associated to a measured lamination compactly supported in $\Sigma_v^\circ$;
\item $\mathrm{Syst}_{\Sigma_v}(L)>0$.
\ee
\end{cor}

Observe that $\partial\calL_\mathrm{Max}(\Sigma,1)=\partial\calL_\mathrm{Hit}(\Sigma,2)$ 
is the Thurston boundary of Teichm\"uller space and consequently the third case of the trichotomy
does not occur in this case.

However, as we will show in a forthcoming paper, for $n\geq2$ the third case of the trichotomy can occur
for length functions in $\partial\calL_\mathrm{Max}(\Sigma,n)$.
To state the corollary concerning these length functions it will be convenient to choose preferred
representatives for elements in $\partial\calL(\Sigma,n)$.
Let $\{c_1,c_2,\dots,c_{6g-6}\}$ be a surface filling family of simple closed geodesics.
We will see that any $[L]\in \overline{\calL(\Sigma,n)}$ satisfies $\sum_{i=1}^{6g-6}L(c_i)>0$
and will henceforth represent a class $[L]$ by the representative satisfying $\sum_{i=1}^{6g-6}L(c_i)=1$.
Let then
\bqn
\Omega(\Sigma,n):=\{[L]\in\partial\calL(\Sigma,n):\mathrm{Syst}(L)>0\}\,.
\eqn
The mapping class group acts on $\P(\R_{\geq0}^\mathscr{C})$
preserving 
$\partial\calL(\Sigma,n)$ and $\Omega(\Sigma,n)$.  
Let $\ell$ be the length function for the hyperbolic metric on $\Sigma=\Gamma\backslash\H^2$.

\begin{cor}\label{cor_intro:1.6}  The set $\Omega(\Sigma,n)$ is an open subset of $\partial\calL(\Sigma,n)$
on which the mapping class group acts properly discontinuously.  
In fact for every $[L]\in\Omega(\Sigma,n)$ there is a neighborhood $V_{[L]}$ of $[L]$
in $\Omega(\Sigma,n)$ and constants $0<C_1\leq C_2$ such that for every $[L']\in V_{[L]}$
\bqn
C_1\ell(c)\leq L'(c)\leq C_2\ell(c)\,.
\eqn
\end{cor}

The above results have also direct applications to the actions of $\pi_1(\Sigma)$ on asymptotic cone actions
built from sequences of maximal or Hitchin representations.
Let $(\rho_k)_{k\in\N}$ be a sequence of representations of $\pi_1(\Sigma)$ into $\mathrm{G}_n(\R)$
that are either Hitchin or maximal, $\omega$ a non-principal ultrafilter on $\N$ and $S$ a finite generating set of $\pi_1(\Sigma)$.
Let $d$ denote the Riemannian distance on $\Xx_n^\R$, let 
$(x_k)_{k\in\N}$ be a sequence of points in $\Xx_n^\R$ 
and set $\lambda_k:=\max_{\gamma\in S}d(\rho_k(\gamma)x_k,x_k)$.
Then the asymptotic cone ${}^\omega\Xx_\lambda$ of the sequence $(\Xx_n^\R,x_k,\frac{d}{\lambda_k})$
can be identified with the building associated to $\mathrm{G}_n({}^\omega\R_\lambda)$,
where ${}^\omega\R_\lambda$ is the Robinson field associated to $\omega$ and the sequence $\lambda=(\lambda_k)_{k\in\N}$.
The $\pi_1(\Sigma)$-action by isometries comes from a representation ${}^\omega\rho_\lambda:\pi_1(\Sigma)\to\mathrm{G}_n({}^\omega\R_\lambda)$
canonically associated to $(\rho_k)_{k\in\N}$, $\omega$ and $\lambda$.  
Then the length function
\bqn
L_{{}^\omega\rho_\lambda}(\gamma):=L({}^\omega\rho_\lambda(\gamma))\,,\quad\text{ for }\gamma\in\pi_1(\Sigma)\,,
\eqn
belongs to $\partial\calL(\Sigma,n)$, \cite{Parreau12};  this, together with Corollaries~\ref{cor_intro:1.5} and \ref{cor_intro:1.6},
the main theorem of \cite{Parreau03} as well as \cite[Proposition~2.2.1 and Lemma~2.0.1]{DGLM},
imply then the following refinement of the decomposition theorem in \cite{BP}:

\begin{cor}\label{cor_intro:1.7}  Let ${}^\omega\rho_\lambda$ be the action of $\pi_1(\Sigma)$ on the asymptotic cone ${}^\omega\Xx_\lambda$
associated as above to a sequence of Hitchin or maximal representations and let $\Sigma=\bigcup_{v\in\Vv_{[l]}}\Sigma_v$
be the decomposition associated to the length function $L_{{}^\omega\rho_\lambda}$.
\begin{itemize}
\item[A.]  Then one of the following holds for $\Sigma_v$:
\begin{enumerate}
\item ${}^\omega\rho_\lambda(\pi_1(\Sigma_v))$ has a global fixed point in ${}^\omega\Xx_\lambda$.
\item $L_{{}^\omega\rho_\lambda}(c)>0$ for every closed geodesic $c\subset\Sigma_v^\circ$ and 
$L_{{}^\omega\rho_\lambda}$ is the length function of a compactly supported measured lamination on $\Sigma_v^\circ$.
\item $\mathrm{Syst}_{\Sigma_v}(L_{{}^\omega\rho_\lambda})>0$
\end{enumerate}
\item[B.] In the case in which the decomposition is trivial, that is $L_{{}^\omega\rho_\lambda}(c)>0$ 
for every geodesic $c\subset\Sigma$, we have the following dichotomy:
\begin{enumerate}
\item Either $L_{{}^\omega\rho_\lambda}$ is the length function of a minimal surface filling measured lamination,
\item or the action ${}^\omega\rho_\lambda$ of $\pi_1(\Sigma)$ on ${}^\omega\Xx_\lambda$ is displacing.
In particular orbit maps are quasi-isometric embeddings. 
\end{enumerate}
\end{itemize}
\end{cor}

\subsection{Outline of the paper}
In Section~\ref{sec:preliminaries} we recall fundamental results of Bonahon recast in our setting.    
Section~\ref{sec:str_geo_curr} is the bulk of the paper.  Given a geodesic current $\mu$, 
we introduce the concept of {\em $\mu$-somewhat short geodesic} and establish
in Section~\ref{subsec:1.0} some basic properties with which we prove in Section~\ref{sec:dec} 
the first part of Theorem~\ref{thm_intro:dec}.
This allows us to define the notion of {\em basic geodesic current}
(Definition~\ref{defn:basic}).
In Section~\ref{subsec:1.2} we show how a $\mu$-somewhat short geodesic leads to a lamination
consisting of $\mu$-somewhat short geodesics (Proposition~\ref{thm:1.5}).
In Section ~\ref{subsec:1.3} we then use a dynamical argument to show that 
the lamination obtained is surface filling.
In Section~\ref{subsec:1.4} we establish a link between vanishing systole and somewhat short geodesics,
using some results of Martone and Zhang, and conclude the proof of Theorem~\ref{thm_intro:main_currents}
 in Section~\ref{subsec:1.5} (see Theorem~\ref{thm:main_currents}).
The case of a thrice punctured sphere requires special treatment in Section~\ref{subsec:thrice}.
In Section~\ref{sec:pos} we establish Theorem~\ref{thm_intro:posSyst} in the introduction; again,
this uses in a crucial way the concept of $\mu$-somewhat short geodesic.
In Section~\ref{sec:5} we prove Corollaries~\ref{cor_intro:1.5} and \ref{cor_intro:1.6}.

\section{Preliminaries on geodesic currents}\label{sec:preliminaries}
Let $\Sigma=\Gamma\backslash \H^2$ be a complete hyperbolic surface of finite area.  
We let $\partial\H^2$ be the boundary of $\H^2$ which is a circle endowed 
with its orientation and corresponding cyclic ordering of triples of points. 
We identify the space of oriented geodesics in $\H^2$ with the space $(\deH^2)^{(2)}$ of pairs of distinct  points. 
A  \emph{geodesic current} is a $\Gamma$-invariant and flip-invariant positive Radon measure on the locally compact space $(\deH^2)^{(2)}$. 
The set $\mathcal C(\Sigma)$ of geodesic currents is then a convex cone in the dual 
of the space of compactly supported functions on $(\deH^2)^{(2)}$; 
the latter is provided with the topology of inductive limit of Banach spaces and $\mathcal C(\Sigma)$ will be equipped with the corresponding weak* topology. 
For $\mu\in\mathcal C(\Sigma)$ we denote by $\supp(\mu)\subseteq (\deH^2)^{(2)}$ its support and call \emph{carrier} of $\mu$, $\carrier(\mu)\subset\Sigma$, 
the closed subset of $\Sigma$ which is the union of all geodesics $g\subset \Sigma$ such that any lift $\wt g\subset \H^2$ lies in $\supp(\mu)$.   
We denote by $\Ccc(\Sigma)$ the set of geodesic currents with compact carrier, 
and by $\CcK(\Sigma)$ the set of those currents with carrier contained in a given compact subset $K\subset \Sigma$. 
For $\gamma\in\G$ hyperbolic let $\g_-$ (resp. $\g_+$) denote the repulsive (resp. attractive) fixed point of $\gamma$ in $\deH^2$. 

\begin{example}\label{ex:2.1}
The following examples of geodesic currents will play an important role in the rest of the paper:
\noindent
\begin{enumerate} 
\item For $\g\in\G$ hyperbolic, we set
$$\delta_\g:=\sum_{\eta\in\G/\<\g\>}\delta_{\eta(\g_-,\g_+)}.$$
Then $\delta_\g$ is a geodesic current whose carrier is the closed geodesic $c\subset \Sigma$ associated to $\g$. We will sometimes denote $\delta_\g$ by $\delta_c$.
\item The Liouville current $\mathcal L$ is the unique (up to positive scaling)  $\PSL(2,\R)$-invariant Radon measure on $(\deH^2)^{(2)}$.  This is $\Gamma$ invariant for every lattice $\Gamma$.
\end{enumerate}
\end{example}

The intersection pairing $i(\mu,\lambda)$ of a current $\mu$ with a current $\lambda$ with compact carrier is defined as follows (see \cite{Bon} for more details): 
the subspace  $\mathcal D\subseteq (\deH^2)^{(2)}\times (\deH^2)^{(2)}$ consisting of pairs of geodesics intersecting in one point is open and 
$\PSL(2,\R)$ acts properly on it: 
indeed the map that associate to a pair $(g,h)$ in $\mathcal D$ the intersection $g\cap h$ gives a $\PSL(2,\R)$-invariant projection of $\mathcal D$ to $\H^2$. 
We restrict the product $\mu\times \lambda$ to $\mathcal D$ and 
define $i(\mu,\lambda)$ as the $(\mu\times \lambda)$-measure of any Borel fundamental domain for the $\G$-action on $\mathcal D$.

\begin{example}
\noindent
\be
\item Given hyperbolic elements $\gamma,\eta\in\G$ representing the closed geodesics $g,h$ in $\Sigma$, the intersection $i(\delta_\g,\delta_\eta)$ is the minimal geometric intersection number between closed curves in the free homotopy class represented by $g$ and $h$. We will often denote it by $i(g,h)$.
\item For an appropriate normalization of the Liouville current $\mathcal L$ we have $i(\mathcal L,\delta_\g)=\ell(\g)$ where $\ell(\g)$ is the hyperbolic length of the closed geodesic represented by $\gamma$.
\ee
\end{example} 
For a general current $\mu$ we will need a formula for the intersection $i(\mu,\delta_\g)$. Given $a,b$ points in $\deH^2$, let $\ioo{a}{b}$ denote the open interval determined by $a,b$, that is
\bq\label{eq:ioo}
\ioo{a}{b}:=\{x\in\deH^2:\,(a,x,b)\text{ is positively oriented}\}.
\eq
The intervals $\ico{a}{b}, \ioc{a}{b}$ and $\icc{a}{b}$ are defined accordingly, so for example
\bqn
\ico{a}{b}:=\{a\}\cup\ioo{a}{b}\,.
\eqn

The following formula for the
intersection of $\mu$ with $\delta_\gamma$ , if $\gamma\in\Gamma$ is a hyperbolic element is well known. A proof can be found in \cite{Martone_Zhang}  for a compact surface $\Sigma$ and goes over verbatim in the case in which  $\Sigma$ has finite area.

\begin{lem}[{\cite[Lemma~4.5]{Martone_Zhang}}]\label{lem:1.21}  If $\gamma\in\Gamma$ is hyperbolic with repulsive and attractive fixed
points $\gamma_-,\gamma_+$ and $z\in\ioo{\gamma_-}{\gamma_+}$, then 
\bqn
i(\mu,\delta_\gamma)=\mu(\ioo{\gamma_+}{\gamma_-}\times\ico{z}{\gamma z})\,.
\eqn
\end{lem}
\noindent In particular $i(\mu,\delta_\gamma)$ is always finite.

One of the most fundamental facts concerning the intersection is the following continuity property due to Bonahon:
\begin{thm}[{\cite[Section 4.2]{B1}}]
For every compact subset $K\subset \Sigma$, the intersection
$$i:\mathcal C(\Sigma)\times \CcK(\Sigma)\to \R_{\geq0}$$
is continuous.
\end{thm}
This continuity property implies a certain number of compactness criteria which will be useful in Section~\ref{sec:pos}. Recall that a current $\mu$ is \emph{surface filling} if every geodesic intersects some geodesic in the support of $\mu$. The Liouville current is surface filling. For an example with compact carrier take
$$\lambda=\sum_{i=1}^{3g-3+p}\delta_{c_i}+\sum_{i=1}^{3g-3+p}\delta_{h_i}$$
where $\{c_1,\ldots,c_{3g-3+p}\}$ and  $\{h_1,\ldots, h_{3g-3+p}\}$ are the simple closed geodesics corresponding to two dual decompositions of $\Sigma$ into pairs of pants.
The arguments in \cite[Proposition 4]{B2} then imply:
\begin{lem}
Let $K\subset \Sigma$ be compact, and $\lambda\in\Ccc(\Sigma)$ be a surface filling current with compact support. Then the following sets are compact:
\begin{enumerate}
\item $\mathcal C(\Sigma)_{\lambda}=\{\mu\in\mathcal C(\Sigma)|\; i(\mu,\lambda)=1\}$;
\item $\CcK(\Sigma)_{\calL}=\{\nu\in\CcK(\Sigma)|\; i(\calL,\nu)=1\}$. Here $\calL$, as in Example \ref{ex:2.1}(2) denotes the Liouville current.
\end{enumerate} 
\end{lem} 
As a corollary one obtains (\cite{B2}):
\begin{prop}
The quotient space $\P\calC(\Sigma)$ of $\calC(\Sigma)^*$ by positive scalar multiplication is compact.
\end{prop}

Recall that a measured lamination on the surface $\Sigma$ is the datum $(\Lambda, m)$ of a closed subset $\Lambda$ of $\Sigma$ foliated by geodesics 
and a measure $m$ on  arcs transverse to $\Lambda$ that is invariant under transverse homotopy.
Let then $(\Lambda,m)$ be a measured lamination with compact support in $\Sigma$, 
and let $(\wt \Lambda,\wt m)$ be its  lift to a measured geodesic  $\G$-invariant lamination in $\H^2$. 
Given any two intervals $\ioo{a}{b}$, $\ioo{c}{d}$ with disjoint closure, there is an open geodesic arc $k\subset \H^2$ such that 
$$\{\l\subset\wt\Lambda|\; |\lambda\cap k|=1\}=\{\lambda\subset \wt\L|\; \lambda\in\ioo{a}{b}\times\ioo{c}{d}\}.$$

Setting $\mu(\ioo{a}{b}\times\ioo{c}{d})=\wt m(k)$ one obtains a geodesic current with $\carrier(\mu)\subset \L$. 
As explained in \cite{B2} this establishes a bijection between measured geodesic laminations with compact support 
and geodesic currents with compact carrier and vanishing self-intersection. We will denote
$$\MLc(\Sigma)=\{\alpha\in\Ccc(\Sigma)|\; i(\alpha,\alpha)=0\}$$
this subspace of currents. We have 
\begin{lem}
\noindent
\be
\item There is a compact subset $K\subset \Sigma$ such that $\MLc(\Sigma)\subset \CcK(\Sigma)$ and consequently $\MLc(\Sigma)$ is a closed subset.
\item For a surface filling current $\lambda$ with compact support
$$\MLc(\Sigma)_{\lambda}:=\MLc(\Sigma)\cap \calC(\Sigma)_\lambda$$
is compact.
\ee
\end{lem}

\begin{defn}\label{defn:systole}  The {\em systole} of the geodesic current $\mu$ is 
\bqn
\Syst(\mu):=\inf\{i(\mu,\delta_\gamma):\,\gamma\in\Gamma\text{ is hyperbolic}\}\,.
\eqn
\end{defn}
\section{Structure of geodesic currents with vanishing systole}\label{sec:str_geo_curr}
The aim of this chapter is to prove Theorem \ref{thm:main_currents}, a refined version of the decomposition theorem for geodesic currents announced in the introduction (Theorem \ref{thm_intro:dec}).  
We fix a geodesic current $\mu$. 
Motivated by the characterization of closed geodesics $c$ whose intersection with a current $\mu$ is trivial (Lemma \ref{lem:1.22}), 
we introduce in \S~\ref{subsec:1.0} the key notion of $\mu$-somewhat short geodesics, 
and establish basic properties about the structure of $\mu$-somewhat short geodesics. 
These results allow us in \S~\ref{sec:dec} to prove Theorem~\ref{thm:decBP} a generalization to the setting of geodesic currents of \cite[Theorem 1.1]{BP}.
In \S~\ref{subsec:1.2} we establish that 
if there is a $\mu$-somewhat short geodesic, there is also a $\mu$-somewhat short geodesic that is simple,
namely whose projection to the surface $\Sigma$ if not self-intersecting.
The closure of the $\Gamma$ orbit of a $\mu$-somewhat short geodesic is then a lamination consisting of $\mu$-somewhat short geodesics.
In \S~\ref{subsec:1.3} we analyse the structure of the lamination obtained in \S~\ref{subsec:1.2} 
and prove the implication \eqref{item:main_currents(2)}$\Rightarrow$\eqref{item:main_currents(3)}
in Theorem~\ref{thm:main_currents}.  In Proposition~\ref{prop:1.23} we adapt certain results of
\cite{Martone_Zhang} to our context and prove the implication \eqref{item:main_currents(1)}$\Rightarrow$\eqref{item:main_currents(2)}.
We also show there  that \eqref{item:main_currents(3)} implies \eqref{item:main_currents(1)}, which is elementary and well known (cfr. for example \cite[Section~6.4]{Morzadec}).

\subsection{Somewhat short geodesics}\label{subsec:1.0}

From Lemma~\ref{lem:1.21} we deduce the following useful characterization of the closed curves $\gamma\subseteq \Sigma$ that have vanishing intersection with the current $\mu$:

\begin{lem}\label{lem:1.22}  The intersection $i(\mu,\delta_\gamma)$ vanishes if and only if $\mu(\ioo{\gamma_-}{\gamma_+}\times\ioo{\gamma_+}{\gamma_-})=0$.
\end{lem}

\begin{proof}  We have that 
\bqn
 i(\mu,\delta_\gamma)= \mu(\ioo{\gamma_+}{\gamma_-}\times\ioc{z}{\gamma z})
=\mu(\ioo{\gamma_+}{\gamma_-}\times\ioc{\gamma^nz}{\gamma^{n+1} z})
\eqn
for every $n\in\Z$.  The lemma follows then from the $\sigma$-additivity of $\mu$:
\bqn
\mu(\ioo{\gamma_-}{\gamma_+}\times\ioo{\gamma_+}{\gamma_-})
=\sum_{n\in\Z}\mu(\ioo{\gamma_-}{\gamma_+}\times\ioc{\gamma^nz}{\gamma^{n+1} z})\,.
\eqn
\end{proof}
A key tool for this paper is the concept of $\mu$-somewhat short geodesic:

\begin{defn}\label{defn: somewhat short}  Let $\mu$ be a geodesic current.
A geodesic $(a,b)\in(\deH^2)^{(2)}$ is {\em $\mu$-somewhat short} if
\bqn
\mu(\ioo{a}{b}\times\ioo{b}{a})=0\,.
\eqn
\end{defn}
\begin{remark}
It follows from Lemma \ref{lem:1.22} that the axis $(\gamma_-,\gamma_+)$ of an hyperbolic element is $\mu$-somewhat short if and only if $i(\mu,\delta_\g)=0$.
\end{remark}

\begin{lem}\label{lem:convergence of ss} Let $\mu$ be a geodesic current.
The subset of $(\deH^2)^{(2)}$ of $\mu$-somewhat short geodesics is closed.
In particular, if $\gamma$ is hyperbolic and $(\gamma_-,b)$ is $\mu$-somewhat short, then $(\gamma_-,\gamma_+)$
is $\mu$-somewhat short.
\end{lem}

\begin{proof}  Assume that $(a_n,b_n)$ is a sequence of $\mu$-somewhat short geodesics

\noindent
\begin{minipage}{.5\textwidth}
converging to the geodesic $(a,b)$.  
Let $x,y,z,t\in\deH^2$ be such that $(x,a,y,z,b,t)$ is positively oriented.
Then for $n$ large enough we have that
$a_n\in \ioo{x}{y}$ and $b_n\in\ioo{z}{t}$. 
\end{minipage}
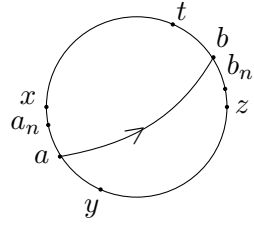
\begin{minipage}{.5\textwidth}
\begin{center}
\begin{tikzpicture}[scale=.6]
\draw (0,0) circle [radius=2];
\filldraw (-1.7,-1.1) circle [radius=1pt] ;
\draw (-2.1,-1.1) node {$a$};
\filldraw (1.7,1.1) circle [radius=1pt] ;
\draw (1.9,1.1) node[above] {$b$};
\draw  (-1.7,-1.1) to [in=240, out=10] (1.7,1.1);
\filldraw (-2,0) circle [radius=1pt];
\draw (-2,0.2) node[left] {$x$};
\filldraw (2,0) circle [radius=1pt];
\draw (1.95,0) node [right]{$z$};
\filldraw (-.8,-1.83) circle [radius=1pt];
\draw (-.6,-1.83) node [below left]{$y$};
\filldraw (.8,1.83) circle [radius=1pt];
\draw (1,1.7) node [above]{$t$};
\filldraw (-1.96,-.4) circle [radius=1pt];
\draw (-1.9, -.4) node [left]{$a_n$};
\filldraw (1.96,.4) circle [radius=1pt];
\draw (2.3,.4) node [above]{$b_n$};
\draw (0,-.57) node[ rotate=35]{$>$};
\end{tikzpicture}
\end{center}
\end{minipage}
This implies that
\bqn
\mu(\ioo{y}{z}\times\ioo{t}{x})\leq\mu(\ioo{a_n}{b_n}\times\ioo{b_n}{a_n})=0\,,
\eqn
that is $\mu(\ioo{a}{b}\times\ioo{b}{a})=0$, by regularity of $\mu$.
\end{proof}

\begin{lem}\label{lem:ss-surj}  Let $\mu$ be a geodesic current. 
Let $(x_1,x_2,x_3,x_4)\in(\deH^2)^4$ be positively oriented, and let us assume
that $(x_1,x_3)$ and $(x_2,x_4)$
are $\mu$-somewhat short.  Then $(x_1,x_2)$, $(x_2,x_3)$, $(x_3,x_4)$ and $(x_4,x_1)$ are 
all $\mu$-somewhat short.  
\end{lem}

\begin{proof}  We show the assertion for $(x_1,x_2)$, the others are analogous.
There are the following relations between intervals:

\noindent
\begin{minipage}{.5\textwidth}
\bqn
\ba
\ioo{x_1}{x_2}&=\ioo{x_1}{x_3}\cap\ioo{x_4}{x_2}\\
\ioo{x_2}{x_1}&=\ioo{x_2}{x_4}\cup\ioo{x_3}{x_1}
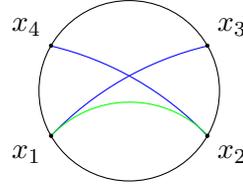
\ea
\eqn
\end{minipage}
\begin{minipage}{.5\textwidth}
\begin{center}
\begin{tikzpicture}[scale=.6]
\draw (0,0) circle [radius=2];
\draw[blue] (-1.732,-1) arc[start angle=136.5, end angle=103.3, radius=7];
\draw[blue] (1.732,-1) arc[start angle=43.5, end angle=76.7, radius=7];
\draw[green] (-1.732,-1) arc[start angle=136.5, end angle=43.5, radius=2.4];
\filldraw (-1.732,-1) circle [radius=1pt] node[below left] {$x_1$};
\filldraw (1.732,-1) circle [radius=1pt] node[below right] {$x_2$};
\filldraw (1.732,1) circle [radius=1pt] node[above right] {$x_3$};
\filldraw (-1.732,1) circle [radius=1pt] node[above left] {$x_4$};
\end{tikzpicture}
\end{center}
\end{minipage}
so that
\bqn
\ioo{x_1}{x_2}\times\ioo{x_2}{x_1}
\subset(\ioo{x_1}{x_3}\times\ioo{x_3}{x_1})\cup (\ioo{x_4}{x_2}\times\ioo{x_2}{x_4})\,,
\eqn
which implies the claim.
\end{proof}

\begin{lem}\label{lem:2.2p} Let $\mu$ be a geodesic current. 
Let $a,b,c,d,e$ be distinct points on $\partial\H^2$ and assume that
\begin{itemize}
\item $(a,b)\pitchfork (c,d)$.
\item $(a,b)$, $(c,d)$, $(e,a)$ and $(e,b)$ are $\mu$-somewhat short.
\end{itemize}
Then $(e,c)$ and $(e,d)$ are $\mu$-somewhat short.
\end{lem}
\begin{proof}  
Consider the ideal triangle with vertices $a,b$ and $e$.  

\medskip
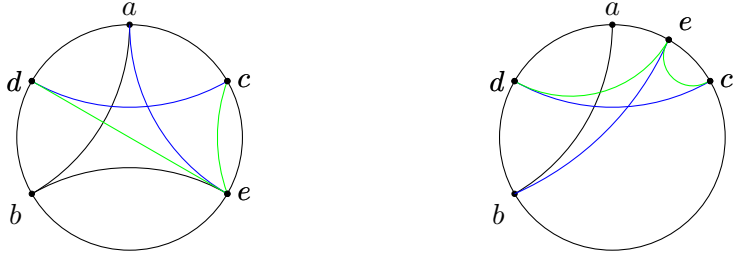
\begin{minipage}{.5\textwidth}
\begin{center}
\begin{tikzpicture}
\draw (0,0) circle [radius=1.5cm];
\filldraw (0,1.5) coordinate(A) circle [radius=1pt] node[above] {$a$};
\filldraw (1.3,.75) coordinate(C) circle [radius=1pt] node[right] {$c$};
\filldraw (-1.3,.75) coordinate(D) circle [radius=1pt] node[left] {$d$};
\filldraw (-1.3,-.75) coordinate(B) circle [radius=1pt] node[below left] {$b$};
\draw (B) arc[start angle=300, end angle=360, radius=2.6];
\draw[blue] (D) arc[start angle=240, end angle=300, radius=2.6];
\filldraw (1.3,-.75) coordinate(E) circle [radius=1pt] node[right] {$e$};
\draw (B) arc[start angle=120, end angle=60, radius=2.6];
\draw[blue] (E) arc[start angle=240, end angle=180, radius=2.6];
\draw[green] (E) -- (D);
\draw[green] (E) arc[start angle=200, end angle=160, radius=2.2];
\filldraw (1.3,.75) coordinate(C) circle [radius=1pt] node[right] {$c$};
\filldraw (-1.3,.75) coordinate(D) circle [radius=1pt] node[left] {$d$};
\filldraw (1.3,-.75) coordinate(E) circle [radius=1pt] node[right] {$e$};
\end{tikzpicture}
\end{center}
\end{minipage}
\begin{minipage}{.5\textwidth}
\begin{center}
\begin{tikzpicture}
\draw (0,0) circle [radius=1.5cm];
\filldraw (0,1.5) coordinate(A) circle [radius=1pt] node[above] {$a$};
\filldraw (1.3,.75) coordinate(C) circle [radius=1pt] node[right] {$c$};
\filldraw (-1.3,.75) coordinate(D) circle [radius=1pt] node[left] {$d$};
\filldraw (-1.3,-.75) coordinate(B) circle [radius=1pt] node[below left] {$b$};
\draw (B) arc[start angle=300, end angle=360, radius=2.6];
\draw[blue] (D) arc[start angle=240, end angle=300, radius=2.6];
\filldraw (E) (.75,1.3) coordinate(E) circle [radius=1pt] node[above right] {$e$};
\draw[blue] (B) arc[start angle=293.8, end angle=336, radius=4];
\draw[green] (D) arc[start angle=240, end angle=330, radius=1.5];
\draw[green] (E) arc[start angle=145, end angle=300, radius=.39];
\filldraw (1.3,.75) coordinate(C) circle [radius=1pt] node[right] {$c$};
\filldraw (-1.3,.75) coordinate(D) circle [radius=1pt] node[left] {$d$};
\filldraw (E) (.75,1.3) coordinate(E) circle [radius=1pt] node[above right] {$e$};
\end{tikzpicture}
\end{center}
\end{minipage}

\medskip
We can assume without loss of generality that  $(a,d,b,c)$ is positively oriented. If the triple $(d,e,c)$ is positively oriented, by Pasch's axiom of geometry
the geodesic $(c,d)$ must cross the geodesic $(e,a)$. 
Then Lemma~\ref{lem:ss-surj}
applied to the four points $a,d,e,c$ implies that $(e,c)$ and $(e,d)$ are $\mu$-somewhat short.

Likewise, if the triple $(d,c,e)$ is positively oriented, the geodesic $(c,d)$ must cross
the geodesic $(e,b)$.   Lemma~\ref{lem:ss-surj} applied now to the four points $b,c,e,d$ 
implies that $(e,c)$ and $(e,d)$ are $\mu$-somewhat short.

\end{proof}

\begin{lem}\label{lem:2.2} Let $\mu$ be a geodesic current.  Let $(g_k)$ be a sequence of $\mu$-somewhat short
geodesics and assume that $g_k\pitchfork g_{k+1}$ for $0\leq k\leq r$.
Let $(x_k,y_k)$ be the endpoints of the geodesics $g_k$ 
chosen in such a way that the quadruple $(x_k,x_{k+1},y_k,y_{k+1})$ has the positive orientation 
for all $0\leq k\leq r$.
Then 
\begin{enumerate}
\item $(x_0,x_k)$ and $(x_0,y_k)$ are $\mu$-somewhat short for all the indices $0\leq k\leq r+1$ 
for which they are defined.
\item Likewise, $(y_0,x_k)$ and $(y_0,y_k)$ are $\mu$-somewhat short for all the indices $0\leq k\leq r+1$ 
for which they are defined.
\end{enumerate}
\end{lem}

\begin{proof} (1) We prove the assertion by induction on $\ell$.

Let us assume that $\ell=0$.  Then, since $(x_0,y_0)\pitchfork(x_1,y_1)$, and $(x_0,y_0)$ and
$(x_1,y_1)$ are $\mu$-somewhat short, it follows from Lemma~\ref{lem:ss-surj} that $(x_0,x_1)$ and $(x_0,y_1)$
are $\mu$-somewhat short.

For the induction step we assume that $(x_0,x_k)$ and $(x_0,y_k)$ are $\mu$-somewhat short for all the $0\leq k\leq\ell\leq r$
and that either one of the following three cases holds:
\begin{enumerate}
\item[(i)] $x_0=x_\ell$ and $(x_0,y_\ell)$ is defined and $\mu$-somewhat short;
\item[(ii)] $x_0=y_\ell$ and $(x_0,x_\ell)$ is defined and $\mu$-somewhat short;
\item[(iii)] $(x_0,x_\ell)$ and $(x_0,y_\ell)$ are both  defined and $\mu$-somewhat short.
\end{enumerate}

Notice that (i) and (ii) cannot occur at the same time, as the geodesic $g_\ell$ has $x_\ell$ and $y_\ell$ as endpoints, 
which therefore must be different from each other.  Notice moreover that if $g_\ell\pitchfork g_{\ell+1}$, the points
$x_\ell,x_{\ell+1},y_\ell$ and $y_{\ell+1}$ must all be distinct.

\medskip
(i) We apply Lemma~\ref{lem:2.2p} to $(a,b,c,d)=(x_\ell, y_\ell, x_{\ell+1},y_{\ell+1})$ and $e=x_0$, to obtain that 
$(x_0,x_{\ell+1})$ and $(x_0,y_{\ell+1})$ are $\mu$-somewhat short.
Moreover now $(x_0,x_{\ell+1})$ and $(x_0,y_{\ell+1})$ are defined.

\medskip
(ii) The same argument as in (i) shows the assertion.

\medskip
(iii) Assume now that $(x_0,x_{\ell})$ and $(x_0,y_{\ell})$ are both defined and $\mu$-somewhat short.
We remarked already that the points  $x_\ell,x_{\ell+1},y_\ell$ and $y_{\ell+1}$ must all be distinct
and we claim that we may assume that the points $x_0,x_\ell,x_{\ell+1},y_\ell$ and $y_{\ell+1}$
are also all distinct.  In fact, by hypothesis $x_0\neq x_\ell$ and $x_0\neq y_\ell$.
If $x_{\ell+1}=x_0$ then $(x_0,y_{\ell+1})=(x_{\ell+1},y_{\ell+1})$ is $\mu$-somewhat short and $(x_0,x_{\ell+1})$ is not defined.
Analogously if $y_{\ell+1}=x_0$ then $(x_0,x_{\ell+1})=(y_{\ell+1},x_{\ell+1})$ is $\mu$-somewhat short and $(x_0,y_{\ell+1})$ is not defined.

Hence if finally the points $x_0,x_\ell,x_{\ell+1},y_\ell$ and $y_{\ell+1}$ are all distinct 
we can apply Lemma~\ref{lem:2.2p} with $(a,c,b,d)=(x_\ell,x_{\ell+1},y_\ell,y_{\ell+1})$
and $e=x_0$ to conclude that $(x_0,x_{\ell+1})$ and $(x_0,y_{\ell+1})$ are $\mu$-somewhat short.

The proof of (2) is completely analogous.
\end{proof}

\begin{prop}\label{prop:L(rho(gamma))=0}  Let $\gamma\in\pi_1(\Sigma)$ be a hyperbolic element and let us assume that 
it is represented by a closed loop $c$ in $\Sigma$ that is the concatenation of segments
belonging to $\mu$-somewhat short geodesics.  Then $(\gamma_-,\gamma_+)$ is $\mu$-somewhat short.
\end{prop}

\begin{proof}
Let $\widetilde c$ be the lift in $\widetilde\Sigma$ of $c$ preserved by $\gamma$ and let $\dots,I_{-1}$, $I_0$, $I_1$, $I_2,\dots$, $I_n=\gamma(I_0)$, $I_{n+1}=\gamma(I_1),\dots$
be the periodic sequence of oriented geodesic segments in $\widetilde c$ such that $\gamma(I_j)=I_{n+j}$ for all $j$.
We define $(x_j, y_j)$ to be the supporting geodesic of the oriented segment $I_j$
in such a way that $(x_j,x_{j+1},y_j,y_{j+1})$ is positively oriented.
By construction $(x_j,y_j)\pitchfork(x_{j+1},y_{j+1})$ and by hypothesis all $(x_j,y_j)$ are $\mu$-somewhat short.  

If any of the $(x_j,y_j)$ or $(y_j,x_j)$ equals $(\gamma_-,\gamma_+)$ we are done.
We hence assume now that $\{\gamma_-,\gamma_+\}\neq\{x_j,y_j\}$ for all $j$.
 In particular we have that either
$x_0\notin\{\gamma_-,\gamma_+\}$ or $y_0\notin\{\gamma_-,\gamma_+\}$.  

We assume without loss of generality that $x_0\notin\{\gamma_-,\gamma_+\}$.  
Then by  Lemma~\ref{lem:2.2}\,(1) we have that $(x_0,x_i)$ is $\mu$-somewhat short whenever is defined, 
which in particular is the case for all $(x_0,\gamma^k x_0)$ for $k\geq1$. 
Hence $(\gamma^\ell x_0,\gamma^{k+\ell}x_0)$ is $\mu$-somewhat short for all $\ell$ and all $k\geq1$.

Either $(\gamma_-,x_0,\gamma_+)$ or $(\gamma_+,x_0,\gamma_-)$ is maximal.
We will only consider the first case, since the second is entirely analogous. 
Pick  $x,y$ so that $(\gamma_-,x,y,\gamma_+)$ is maximal and let $\ell$ and $k\geq1$ be integers
such that the 6-tuple $(\gamma_-,\gamma^\ell x_0,x,y,\gamma^{\ell+k}x_0,\gamma_+)$ is maximal.
Since $(\gamma^\ell x_0,\gamma^{\ell+k} x_0)$ is $\mu$-somewhat short, then
$\mu(\ioo{\gamma_-}{\gamma_+}\times\ioo{x}{y})=0$
Since $x$ and $y$ were arbitrary, we deduce that $(\gamma_-,\gamma_+)$ is $\mu$-somewhat short.
\end{proof}

\subsection{A decomposition theorem for geodesic currents}\label{sec:dec}
Let $\mu\in\calC(\Sigma)$ be a geodesic current, where $\Sigma=\G\backslash\H^2$ is a finite area hyperbolic surface.
We consider
$$\calG_\mu(0)=\{c\subset \Sigma \text{ closed geodesic }|\;i(\mu,c)=0\}.$$
On $\calG_\mu(0)$ we put the graph structure $c_1\equiv c_2$ if $i(c_1,c_2)>0$. 
Notice that there is a loop at $c\in \calG_\mu(0)$ if and only if $c$ is self-intersecting. 
Let $\mathfrak C\subset\calG_\mu(0)$ be a connected component  which is not reduced to a single simple closed curve. 
Let $c_1,c_2,\ldots$ be an enumeration of the vertices of $\mathfrak C$  
such that $c_i\equiv c_{i+1}$ for every $i\geq 1$, and choose $x\in c_1$ a basepoint.

Let $X_n=\bigcup_{i=1}^n c_i$. Then $X_n$ is a connected graph. Let $\G_n$ be the image of the morphism
$$\pi_1(X_n,x)\to \pi_1(\Sigma,x).$$
Then by Proposition \ref{prop:L(rho(gamma))=0} every hyperbolic $\g\in\G_n$ satisfies $i(\mu,\delta_\g)=0$.  

Let $\Sigma_n$ be the totally geodesic subsurface obtained by taking a regular neighbourhood of $X_n$, 
adding to it the connected components of the complement that are either simply connected or containing a single cusp, and straightening the boundary components. 
In fact $\Sigma_n$ is the projection of the closet convex hull of the limit set of $\G_n$ and $\G_n=\pi_1(\Sigma_n,x)$. 
Since $\Sigma$ and $\Sigma_n$ are of finite topological type, the sequence $\G_n\subseteq \G_{n+1}$ stabilizes, and we set $\Sigma^\mathfrak C:=\Sigma_n$ for $n$ large enough.

Clearly each boundary component $c\subset \de\Sigma^\mathfrak C$ has the property that 
if a closed geodesic $c'\subset \Sigma$ is such that $i(c,c')>0$ then $i(\mu,c')>0.$ Thus we have a decomposition
$$\Sigma =\bigcup_{v\in \mathcal{V}_\mu}\Sigma_v=\bigcup_{v\in \mathcal{V}_{\mu,S}}\Sigma_v\cup\bigcup_{v\in \mathcal{V}_{\mu,P}}\Sigma_v$$  
such that if $v\in \mathcal{V}_{\mu,S}$, then every closed curve $c\subset\Sigma_v$ is $ \mu$-somewhat short, 
and if $v\in \mathcal{V}_{\mu,P}$ then every $c\subset \Sigma_v$ not boundary parallel has $i(\mu,c)>0$, 
namely it has positive intersection (therefore the choice of the indices, S for $\mu$-somewhat short, P for positive intersection). 
Let $\mathcal{E}$ be the set of boundary components of all the subsurfaces. 
Observe that $\mathcal{E}$ coincides with the set of \emph{special} geodesics defined in the introduction, 
namely those geodesics $c$ that are $\mu$-somewhat short and have the additional property that $i(\mu,c')>0$ if $i(c, c')>0$.

Now let $g\subset \H^2$ be a geodesic such that $\p_\Sigma(g)$ intersects $\mathring \Sigma_v$, where $v\in\mathcal{V}_{\mu,S}$. 
We choose a connected component $\wt \Sigma_v$ of $\p_\Sigma^{-1}(\Sigma_v)$ such that $g$ intersects $\mathring{\wt\Sigma}_v$. 
There are two cases: either $g$ intersects a boundary component of $\wt\Sigma_v$ 
and hence $g$ does not belong to the support of $\mu$ or $g\subseteq{\wt \Sigma}_v$. 
In this case $g$ cannot be a boundary component of $\wt\Sigma_v$ and hence must intersect transversely some axis of an element in $\pi_1(\Sigma_v,x)$. 
This implies that $g$ doesn't belong to the support of $\mu$. 
Thus, if $g\subset \H^2$ is in the support of $\mu$, then there is $v\in \mathcal{V}_{\mu,P}$ with $\p_\Sigma(g)\subseteq \Sigma_v$.

  For each $v\in \mathcal{V}_{\mu,P}$, define 
  $$\mathcal G_v=\{g\subset \H^2|\; \p_\Sigma(g)\subset\mathring \Sigma_v\}.$$
Then $\mathcal G_v\subset(\deH^2)^{(2)}$ is a $\G$-invariant Borel subset and 
$$\supp(\mu)\subset\bigcup_{v\in V_P}\mathcal G_v\cup\bigcup_{c\in E}\{g|\;\p_\Sigma(g)=c\}.$$
Let $\chi_{v}$ be the characteristic function of $\mathcal G_v$ and set $\mu_v:=\mu\cdot\chi_v$. Then $\mu_v\in\mathcal C(\Sigma)$ and we have proven:

\begin{thm}\label{thm:decBP}
There are numbers $\lambda_c\in\R$ with
$$\mu=\sum_{v\in \mathcal{V}_\mu} \mu_v+\sum_{c\in \mathcal{E}_\mu}\lambda_c\delta_{c}\,,$$
where $\mu_v$ is supported in $\overline{\mathcal G_v}$ and $\delta_{c}$ is the geodesic current associated to the closed geodesic $c$. 
\end{thm}

Theorem \ref{thm:decBP} motivates the following definition:
\begin{defn}\label{defn:basic}
A geodesic current $\mu$ is \emph{basic} if there exists a  subsurface with totally geodesic boundary $\Sigma_\mu\subseteq\Sigma$ such that
\begin{enumerate}
\item $\carrier(\mu)\subseteq \Sigma_\mu$;
\item for every closed geodesic $c\subset\mathring\Sigma_\mu$.  $i(\mu,\delta_c)>0$; 
\item for all  $\g\in\G$ representing a boundary component of $\Sigma_\mu$, $\mu(\{(\gamma_-,\gamma_+)\})=0$.
\end{enumerate}
\end{defn}
\begin{remark}
Observe that if the geodesic current $\mu$ is basic then for all  $\g\in\G$ representing a boundary component of $\Sigma_\mu$, $(\gamma_-,\gamma_+)$ is $\mu$-somewhat short, since no geodesic in the support of $\mu$ can intersect the lift $(\gamma_-,\gamma_+)$ of a boundary component.
\end{remark}
Observe that if $\mu$ is a basic current, then the \emph{supporting surface} $\Sigma_\mu$ of $\mu$ is the smallest subsurface of $\Sigma$ with totally geodesic boundary that contains $\carrier(\mu)$. 

We fix a component $\wt\Sigma_\mu$ of $\p_\Sigma^{-1}(\Sigma_\mu)$ and denote by $\Gamma_\mu$ the  stabilizer in $\G$ of $\wt\Sigma_\mu$. We say that an element $\gamma\in\G_\mu$ is \emph{peripheral} if either $\gamma$ is parabolic or it corresponds to a boundary component of $\Sigma_\mu$.
As a refinement of Definition \ref{defn:systole}, we set
\begin{defn}\label{defn:subsystole}  The {\em systole} of a basic geodesic current $\mu$ restricted to its supporting subsurface is 
\bqn
\Syst_{\Sigma_{\mu}}(\mu):=\inf\{i(\mu,\delta_\gamma):\,\gamma\in\Gamma_\mu\text{ not peripheral}\}.
\eqn
\end{defn}

\subsection{From somewhat short geodesics to laminations}\label{subsec:1.2}

We say that a geodesic $(a,b)$ is {\em simple} if $\p_\Sigma((a,b))$ is not self-intersecting.  
In this subsection we focus on a basic geodesic current $\mu$, 
and show that, if there exists a $\mu$-somewhat short geodesic $(a,b)$, 
then there exists a lamination $\wt \Lambda$ consisting of $\mu$-somewhat short geodesics. 
This is achieved in two steps: in Proposition \ref{lem:simple_ss_geod} we show that we can assume that $(a,b)$ is simple, 
in Proposition \ref{thm:1.5} we show that the closure of a simple somewhat short geodesic is a lamination consisting of somewhat short geodesics
and with at most two minimal sublaminations.

Let $(a,b)$ be $\mu$-somewhat short and consider $\p_\Sigma\left((a,b)\right)\subset\Sigma$ as a pathwise connected subspace.
Fix $\ast\in\p_\Sigma\left((a,b)\right)$ and observe that every element in $\pi_1(\p_\Sigma\left((a,b)\right),\ast)$
can be represented by a concatenation of geodesic segments whose lifts lie on a $\mu$-somewhat short geodesic.

\begin{prop}\label{lem:simple_ss_geod} Let $\mu$ be a basic geodesic current and let $(a,b)$ be a $\mu$-somewhat short geodesic whose projection is contained in the supporting surface $\Sigma_\mu$. Then:
\be
\item either the image of $\pi_1(\p_\Sigma((a,b)),\ast)\to\pi_1(\Sigma,\ast)$ is $\{e\}$, in which case $(a,b)$ is simple, or
\item its image is $\< \gamma\>$, where $\g$ is peripheral.
In either case neither $a$ nor $b$ are fixed by $\gamma$ and there is $k\in\Z$ such that $(\g^ka,b)$ is $\mu$-somewhat short and simple.
\ee
\end{prop}

\begin{proof} It follows from Proposition~\ref{prop:L(rho(gamma))=0} that the image $\p_\Sigma\left((a,b)\right)\subset\Sigma$ does not contain
any non-peripheral element.  Hence is must be either $\{e\}$ or $\<\gamma\>$ with $\gamma$ peripheral.  
In the first case $\p_\Sigma\left((a,b)\right)$ cannot self-intersect since a geodesic monogon is always $\pi_1$-non-trivial.  This proves (1).

In the second case, let $\infty$ be a fixed point of $\gamma$. Since $(a,b)$ is self intersecting, neither $a$ nor $b$ are fixed points of $\gamma$, and,  if $\gamma$ is hyperbolic, $(a,b)$ doesn't intersect the axis of $\gamma$. In particular, up to switching $\gamma$ with $\gamma^{-1}$ we can assume that $(a,\gamma a, b,\g b)$ is positively oriented.

Let $k\geq 1$ be maximal such that $(a,\g^ka,b,\g^k b)$ is positively oriented.
By Lemma \ref{lem:ss-surj}   $(\gamma^k a,b)$ is $\mu$-somewhat short. We will prove by contradiction that $\p_\Sigma\left((\gamma^ka,b)\right)$ is simple.
%
We denote by $\eta\in\Gamma$ an element $\eta\neq e$ such that
\bq\label{eq:self-intersection}
\eta(\gamma^k a,b)\pitchfork(\gamma^k a,b)\,,
\eq
and we denote by $T$ the triangle bounded by the geodesics $(\g^k a,b)$, $(\g^ka,\g^kb)$ and $(a,b)$.
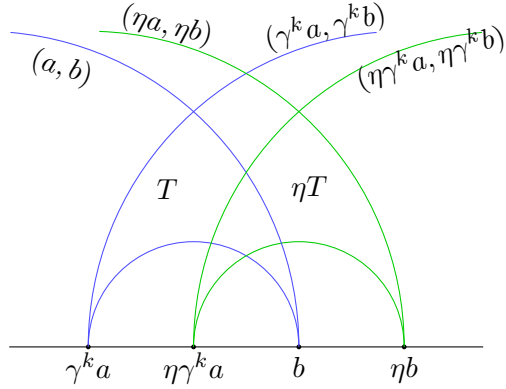
\begin{center}
\begin{tikzpicture}[scale=.7]
\draw (-4.5,0) -- (4.5,0);
\draw[color=blue!70] (-3,0) arc [start angle=180, end angle=0, radius=2];
\draw[green!80!black] (-1,0) arc [start angle=180, end angle=0, radius=2];
\filldraw (-3,0) coordinate (A) circle [radius=1pt];
\draw (-3, -.4) node {$\g^ka$};
\filldraw (-1,0) coordinate (etaA) circle [radius=1pt];
\draw (-1,-.4) node {$\eta \g^ka$};
\filldraw (1,0) coordinate (nkB) circle [radius=1pt] node[below] {$b$};
\filldraw (3,0) coordinate (etankB) circle [radius=1pt] node[below] {$\eta  b$};
\draw[color=blue!70] (A) arc [start angle=180, end angle =95, radius=6];
\draw (1.5,6.1) node[rotate=12] {$(\g^ka,\g^kb)$};
\draw[color=blue!70] (nkB) arc [start angle=0, end angle=85, radius=6];
\draw (-3.5,5.3) node[rotate=-17] {$(a,b)$};
\draw[green!80!black] (etaA) arc [start angle=180, end angle =95, radius=6];
\draw (3.5,5.5) node[rotate=17] {$(\eta\g^k a,\eta\g^k b)$};
\draw[green!80!black] (etankB) arc [start angle=0, end angle=88, radius=6];
\draw (-1.5,6.1) node[rotate=-14] {$(\eta a,\eta b)$};
\draw (-1.5,3) node {$T$};
\draw (1.2,3) node {$\eta T$};

\end{tikzpicture}
\end{center}
Since $(\g^ka,b)\pitchfork \eta(\g^ka,b)$, the geodesic $(\g^ka,b)$ intersects one of the sides of the triangle $\eta T$.
By Pasch's theorem, $(\g^ka,b)$ must intersect also another of the geodesic sides of $\eta T$,
call it $\ell$.  But this means that $\ell$ intersects one of the sides of $T$.  
Thus, again by Pasch's theorem, it must intersect one of the other sides, in particular  
either $\eta (a, b)$ or $\eta (\gamma^ka, \g^kb)$ intersects either $(a,b)$ or $(\g^ka,\g^kb)$.
This implies that $\eta\in\<\gamma\>$, which contradicts the maximality of $k$. Indeed, since $(\g^ka,b)\pitchfork \eta(\g^ka,b)$, either 
\be
\item[(i)] $\g^k a<\eta b< b$, and hence $\eta^{-1}\g^k$ is a bigger power of $\g$ with the same defining property of $\g^k$ or 
\item[(ii)] $\g^k a<\eta\g^k a<b$ and $\eta\g^k$ is a bigger power of $\g$ with the same defining property of $\g^k$. 
\ee
\end{proof}
The following example illustrates the geometric content of Proposition \ref{lem:simple_ss_geod}:
\begin{example}   
This is a biinfinite geodesic connecting the cusps $c_1,c_2$
with a subpath homotopic to the cusp $c_3$.  The corresponding simple geodesic
just connects $c_1$ to $c_2$. 

\medskip
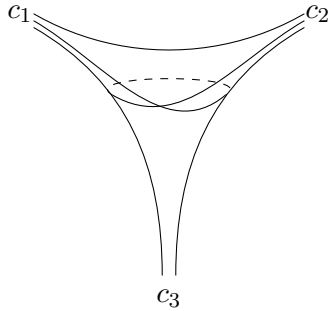
\begin{center}
\begin{tikzpicture}[scale=.6]
\draw[xshift=0, yshift=10cm, domain=240:300] plot(\x:6);
\draw (.15,-1) to [out=90, in=210] (3,4.5);
\draw (-.15,-1) to [out=90, in=330] (-3,4.5);
\draw (-3.3,4.8) node {$c_1$};
\draw (3.3,4.8) node {$c_2$};
\draw (0,-1.5) node {$c_3$};
\draw (-3,4.65) to [out=-30, in=220] (1.25,3);
\draw (3,4.65) to [out=210, in=-30] (-1.25,3);
\draw[dashed] (-1.25,3) to [out=150, in=30] (1.25,3);
\end{tikzpicture}
\end{center}
\end{example}

\begin{prop}\label{thm:1.5}  Let $\mu$ be a basic geodesic current.
If $(a,b)$ is $\mu$-somewhat short and simple, then $\widetilde\Lambda:=\overline{\Gamma(a,b)}$ is  a $\Gamma$-invariant geodesic
lamination on $\H^2$ consisting of $\mu$-somewhat short geodesics.  
Its projection $\Lambda$ has at most one isolated leaf, which is $\p_\Sigma\left((a,b)\right)$, and at most
two minimal sublaminations.
\end{prop}
\begin{proof}
  Let us set $\widetilde\Lambda:=\overline{\Gamma(a,b)}$, which is clearly a closed subset of $\H^2$. Since for two geodesics to intersect in $\H^2$ is an open condition, $\widetilde\Lambda$ is  disjoint union of geodesics and hence, by definition, a lamination.

The fact that $\Lambda$ has at most one isolated leaf and almost two minimal components 
follows from the general structure theory of laminations. 
Since $(a,b)$ is $\mu$-somewhat short, it follows from Lemma~\ref{lem:convergence of ss}
that every  geodesic in $\wt\Lambda$ is $\mu$-somewhat short.  
\end{proof}
Observe that, in general, the lamination $\wt \Lambda$ might be very degenerate:
if $(a,b)$ is $\mu$-somewhat short and both $a$ and $b$ are cusps, then the geodesic $\lambda=\p_\Sigma\left((a,b)\right)$
is asymptotic to two cusps.  In this case the $\mu$-somewhat short and simple geodesic $(a,b')$ provided
by Proposition~\ref{lem:simple_ss_geod} projects then to a non-self-intersecting geodesic $\lambda'$ that is asymptotic to $\lambda$.
Thus $\lambda'$ is a properly  embedded copy of $\bR$, hence a closed subset. 
In order to avoid this case, we have the following definition: 
\begin{defn}
A geodesic $(a,b)$ is \emph{recurrent} in $\Sigma_\mu$ if at least one between $a$ and $b$ is not a fixed point of a peripheral element in $\Gamma_\mu$.
\end{defn}
Observe that a geodesic $(a,b)$ is recurrent if and only if it is recurrent in the open surface $\mathring\Sigma_\mu$, namely if there are compact subsets in the interior of $\Sigma_\mu$ that are visited infinitely often.  

The aim of the rest of the chapter is to prove the following refinement of Theorem \ref{thm_intro:dec} from the introduction:

\begin{thm}\label{thm:main_currents}  Let $\mu$ be a basic geodesic current.  Then the following are equivalent:
\be
\item\label{item:main_currents(1)} $\Syst_{\Sigma_\mu}(\mu)=0$;
\item\label{item:main_currents(2)} There exists a recurrent geodesic $(a,b)$ that is $\mu$-somewhat short;
\item\label{item:main_currents(3)} The support of $\mu$ is a $\Gamma$-invariant lamination in $\H^2$
whose projection on $\mathring\Sigma_\mu$ is compactly supported, minimal and surface filling.
\ee
\end{thm}

We say that a lamination is {\em surface filling} if the complementary regions in $\Sigma_\mu$ are only ideal polygons, ideal polygons containing one cusp and ideal polygons containing one boundary component.

We will also establish that, under the equivalent conditions of Theorem~\ref{thm:main_currents}, the function
\bqn
(a,b,c,d)\mapsto\mu(\ioo{a}{b}\times \ioo{c}{d})
\eqn
is continuous, provided $\icc{a}{b}$ and $\icc{c}{d}$ are disjoint.

\begin{rem}
It is possible to show, using Lemma \ref{lem:2.2} and Lemma \ref{lem:convergence of ss} that if $\mu$ is a basic geodesic current, then there cannot exist a $\mu$-somewhat short geodesic $(a,b)$ connecting two boundary components of the supporting surface $\Sigma_v$: otherwise the boundary components of the pair of pants determined by $(a,b)$ would be $\mu$-somewhat short.
\end{rem}
\subsection{Minimal laminations and complementary regions}\label{subsec:1.3}
Recall that if $\Lambda\subset\Sigma=\Gamma\backslash\H^2$ is a compactly supported geodesic lamination 
in a complete finite area hyperbolic surface, the complement $\Sigma\smallsetminus\Lambda$ 
is a finite union of components of the following types \cite[Theorem I.4.2.8]{notesonnotes}:
\be
\item an ideal polygon;
\item an ideal polygon containing one cusp;
\item a totally geodesic subsurface with geodesic boundary to which one has added a \emph{crown} for each boundary geodesic (such a subsurface can possibly be reduced to a single geodesic).
\ee
\noindent A crown is an infinite cylinder bounded by finitely many ideal sides.
A compactly supported geodesic lamination $\Lambda\subset\Sigma$ \emph{fills} a subsurface $\Sigma'\subseteq \Sigma$ with totally geodesic boundary if, for every closed curve $c\subset \mathring{\Sigma}'$, $c$ intersects some leaf of $\Lambda$. Observe that a measured geodesic lamination $(\Lambda, m)$, when regarded as a geodesic current, never fills a surface $\Sigma$, not even when the lamination $\Lambda$ does fill. A geodesic lamination $\Lambda$ is \emph{minimal} if every half ray of a geodesic belonging to $\Lambda$ is dense in $\Lambda$.

A crucial step in the proof of the implication $(2)\Rightarrow (3)$ in Theorem \ref{thm:main_currents} consists of showing that any geodesic bounding a crown of a $\mu$-somewhat short lamination $\wt \Lambda$ is itself $\mu$-somewhat short (Proposition \ref{thm:1.13} below). A necessary condition for this to be true is expressed in the following lemma:


\begin{lem}\label{lem:1.14}  Let $\mu$ be a geodesic current and assume that 
there exists a $\Gamma$-invariant lamination $\wt \Lambda$ consisting of $\mu$-somewhat short geodesics without isolated leaves. 
Let $a,b,c$ be consecutive vertices of a complementary region $\mathcal R$ of $\wt \Lambda$ 
labelled in such a way that $(c,b,a)$ is positively oriented.  Then 
\bqn
\mu(\{b\}\times\icc{a}{c})=0\,.
\eqn
\end{lem}

\begin{proof} 
Let $\gamma\in\Gamma$ be an hyperbolic element whose axis $(\gamma_-,\gamma_+)$
crosses  $(a,b)$ and $(b,c)$ and $(\g^-,a,c,\g^+,b)$ is positively oriented. 
Denote by $p_0$ the intersection $(\g_-,\g_+)\cap(a,b)$ and 
focus on the subsegment $I$ of the axis of $\gamma$ with endpoints  $\g^{-1}p_0, p_0$. 
Denote by $L$ the hyperbolic distance $L=d(\g^{-1}p_0, p_0)$ and choose a point $p$ far enough in the ray from $p_0$ to $b$ 
so that its distance from any side of $\Rr$ different from $(a,b)$ or $(b,c)$ is bigger than $2L$.
Denote by $r$ the ray from $p$ to $b$.

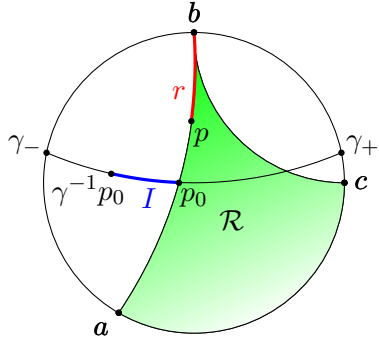
\begin{center}
\begin{tikzpicture}
\draw (0,0) circle [radius=2];
\filldraw (0,2) coordinate(B) circle [radius=1pt] node[above] {$b$};
\filldraw (-1,-1.73) coordinate(A) circle [radius=1pt] node[below left] {$a$};
\filldraw (2,0) coordinate(C) circle [radius=1pt] node[right] {$c$};
\draw (B) arc[start angle=180, end angle=270, radius=2];
\draw (A) arc[start angle=326, end angle=360, radius=6];
\shadedraw[top color=green, bottom color=white, draw=black, shading=axis,shading angle=20]
  (A) arc[start angle=326, end angle=360, radius=6] --
  (B) arc[start angle=180, end angle=270, radius=2] --
  (C) arc[start angle=360, end angle=240, radius=2] -- cycle;
\filldraw (0,2) coordinate(B) circle [radius=1pt] node[above] {$b$};
\filldraw (-1,-1.73) coordinate(A) circle [radius=1pt] node[below left] {$a$};
\filldraw (2,0) coordinate(C) circle [radius=1pt] node[right] {$c$};  
\filldraw (-1.96,0.4) circle [radius=1pt];
\draw (-2.246, .5) node {$\gamma_-$};
\filldraw (1.96,.4) circle [radius=1pt];
\draw (2.246, .5) node {$\gamma_+$};
\draw[xshift=0, yshift=5cm, domain=247:293] plot(\x:5);
\draw[red, very thick, xshift=-6cm, yshift=1.6cm, domain=352.5:364] plot(\x:6.015);
\draw[red] (-.2,1.2) node {$r$};
\filldraw (0,2) coordinate(B) circle [radius=1pt] node[above] {$b$};
\filldraw (-1,-1.73) coordinate(A) circle [radius=1pt] node[below left] {$a$};
\filldraw (2,0) coordinate(C) circle [radius=1pt] node[right] {$c$};  
\filldraw (-.035,.82) circle [radius=1pt];
\draw (.1,.6) node {$p$}; 
\draw[blue, very thick, xshift=0, yshift=5cm, domain=257:268] plot(\x:5);
\draw[blue, very thick] (-.6,-.2) node {$I$}; 
\filldraw (-.2,0) circle [radius=1pt];
\draw (0,.-.2) node {$p_0$}; 
\filldraw (-1.1,.12) circle [radius=1pt];
\draw (-1.4,-.1) node {$\gamma^{-1}p_0$};
\draw (.5,-.5) node {$\Rr$};
\end{tikzpicture}
\end{center}

Since $\wt \Lambda$ has no isolated leaves,  $\p_\Sigma(r)$ is dense in a minimal component of $\Lambda:=\p_\Sigma(\wt\Lambda)$ and in particular meets the transversal
$\p_\Sigma(I)$ in infinitely many points.  
Thus, moving to the universal covering $\H^2$ and choosing appropriate representatives,
we find a sequence of points $p_n\in r$ and a sequence of elements $\gamma_n\in\Gamma$ such that  $\gamma_n(p_n)\in I$ for every $n\geq1$.

Since any point $p_n$ in $r$ has distance more than $2L$ from any side of $\Rr$ different from $(a,b)$ or $(b,c)$ 
and $\g_n\Rr\cap I $ is contained in $I$, 
we get that $\g_n\Rr\cap I $  is an interval bounded by $\g_n(p_n)$ and $\g_n(q_n)$ for some $q_n$ in $(b,c)$.
Observe that the $\g_n\Rr$ are pairwise distinct: 
indeed for all $\gamma\in \Gamma$, the intersection $\gamma\Rr\cap I$ is connected, if not empty. 
This implies that, if $\gamma_n\Rr=\gamma_m\Rr$ then  $\g_n b=\gamma_m b$ 
which implies that $\g_n=\g_m$ since $\widetilde\Lambda$ has no isolated leaves, and hence its endpoint at infinity have trivial stabilizer.

Using the natural ordering on $I$,
we have for each $n$, either
\bqn
\gamma_n(p_n)<\gamma_n(q_n) \quad \text{ or }\quad \gamma_n(q_n)<\gamma_n(p_n)\,.
\eqn
Passing to a subsequence we may assume that $\gamma_n(p_n)<\gamma_n(q_n)$ for all $n\geq1$.

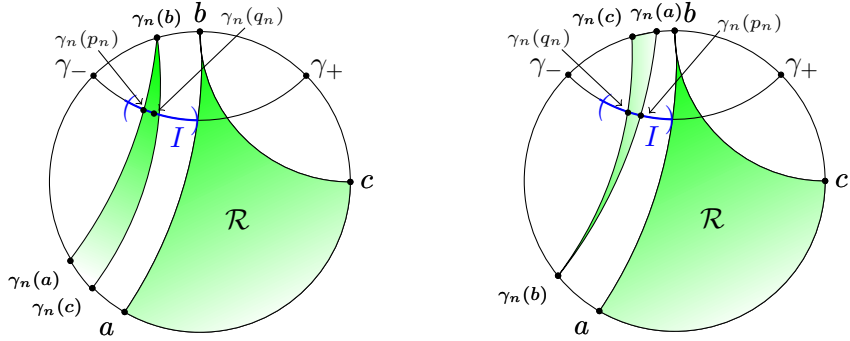
\begin{minipage}{.5\textwidth}
\begin{center}
\begin{tikzpicture}
\draw (0,0) circle [radius=2];
\filldraw (0,2) coordinate(B) circle [radius=1pt] node[above] {$b$};
\filldraw (-1,-1.73) coordinate(A) circle [radius=1pt] node[below left] {$a$};
\filldraw (2,0) coordinate(C) circle [radius=1pt] node[right] {$c$};
\draw (B) arc[start angle=180, end angle=270, radius=2];
\draw (A) arc[start angle=326, end angle=360, radius=6];
\shadedraw[top color=green, bottom color=white, draw=black, shading=axis, shading angle=20]
  (A) arc[start angle=326, end angle=360, radius=6] --
  (B) arc[start angle=180, end angle=270, radius=2] --
  (C) arc[start angle=360, end angle=240, radius=2] -- cycle;
\filldraw (0,2) coordinate(B) circle [radius=1pt] node[above] {$b$};
\filldraw (-1,-1.73) coordinate(A) circle [radius=1pt] node[below left] {$a$};
\filldraw (2,0) coordinate(C) circle [radius=1pt] node[right] {$c$};  
\filldraw (-.57,1.92) coordinate(GNB) circle [radius=1pt] node[above] {\tiny{$\gamma_n(b)$}};
\filldraw (-1.72,-1.05) coordinate(GNA) circle [radius=1pt]  node[below left] {\tiny{$\gamma_n(a)$}};
\filldraw (-1.43,-1.412) coordinate(GNC) circle [radius=1pt] node[below left] {\tiny{$\gamma_n(c)$}};
\shadedraw[top color=green, bottom color=white, draw=black, shading=axis,shading angle=-10]
  	(GNB) arc[start angle=354.5, end angle=323, radius=5.8] --
  	(GNA) arc[start angle=210, end angle=224, radius=2] --
	(GNC) arc[start angle=323, end angle=368, radius=4.5] -- cycle;
\filldraw (-.57,1.92) coordinate(GNB) circle [radius=1pt] node[above] {\tiny{$\gamma_n(b)$}};
\filldraw (-1.72,-1.05) coordinate(GNA) circle [radius=1pt]  node[below left] {\tiny{$\gamma_n(a)$}};
\filldraw (-1.43,-1.412) coordinate(GNC) circle [radius=1pt] node[below left] {\tiny{$\gamma_n(c)$}};
\filldraw (-1.414,1.414) circle [radius=1pt];
\draw (-1.7, 1.5) node {$\gamma_-$};
\filldraw (1.414,1.414) circle [radius=1pt];
\draw (1.7, 1.5) node {$\gamma_+$};
\draw[xshift=0, yshift=2.82cm, domain=225:315] plot(\x:2);
\draw[blue, thick, xshift=0, yshift=2.82cm, domain=240:269] plot(\x:2);
\draw[blue, very thick] (-.06,.82) node {$)$};
\draw[blue, very thick] (-.95,1) [rotate=0]  node [rotate=-20] {$($};
\draw[blue, very thick] (-.3,.6) node {$I$}; 
\filldraw (-.75,.955) circle [radius=1pt];
\draw (-1.5,1.9) node {\tiny{$\gamma_n(p_n)$}};
\draw[->] (-1.35,1.75) -- (-.77,1.03);
\filldraw (-.61,.91) circle [radius=1pt];
\draw (.7,2.2) node {\tiny{$\gamma_n(q_n)$}};
\draw[->] (.6,2.05) -- (-.55,.95);
\draw (.5,-.5) node {$\Rr$};
\end{tikzpicture}
\end{center}\end{minipage}
\begin{minipage}{.5\textwidth}
\begin{center}
\begin{tikzpicture}
\draw (0,0) circle [radius=2];
\filldraw (0,2) coordinate(B) circle [radius=1pt];
\draw (0.2,2.25) node {$b$};
\filldraw (-1,-1.73) coordinate(A) circle [radius=1pt] node[below left] {$a$};
\filldraw (2,0) coordinate(C) circle [radius=1pt] node[right] {$c$};
\draw (B) arc[start angle=180, end angle=270, radius=2];
\draw (A) arc[start angle=326, end angle=360, radius=6];
\shadedraw[top color=green, bottom color=white, draw=black, shading=axis,shading angle=20]
  (A) arc[start angle=326, end angle=360, radius=6] --
  (B) arc[start angle=180, end angle=270, radius=2] --
  (C) arc[start angle=360, end angle=240, radius=2] -- cycle;
\filldraw (0,2) coordinate(B) circle [radius=1pt];
\draw (0.2,2.25) node {$b$};
\filldraw (-1,-1.73) coordinate(A) circle [radius=1pt] node[below left] {$a$};
\filldraw (2,0) coordinate(C) circle [radius=1pt] node[right] {$c$};  
\filldraw (-1.55,-1.26) coordinate(GNB) circle [radius=1pt] node[below left] {\tiny{$\gamma_n(b)$}};
\filldraw (-.56,1.92) coordinate(GNC) circle [radius=1pt] node[above left] {\tiny{$\gamma_n(c)$}};
\filldraw (-.24,1.99) coordinate(GNA) circle [radius=1pt] node[above] {\tiny{$\gamma_n(a)$}};
\shadedraw[top color=green, bottom color=white, draw=black, shading=axis,shading angle=90]
(GNA) arc[start angle=354.5, end angle=322, radius=6.2] --
(GNB) arc[start angle=323, end angle=362.7, radius=4.9] -- cycle;
\filldraw (-1.55,-1.26) coordinate(GNB) circle [radius=1pt] node[below left] {\tiny{$\gamma_n(b)$}};
\filldraw (-.56,1.92) coordinate(GNC) circle [radius=1pt] node[above left] {\tiny{$\gamma_n(c)$}};
\filldraw (-.24,1.99) coordinate(GNA) circle [radius=1pt] node[above] {\tiny{$\gamma_n(a)$}};
\filldraw (-1.414,1.414) circle [radius=1pt];
\draw (-1.7, 1.5) node {$\gamma_-$};
\filldraw (1.414,1.414) circle [radius=1pt];
\draw (1.7, 1.5) node {$\gamma_+$};
\draw[xshift=0, yshift=2.82cm, domain=225:315] plot(\x:2);
\draw[blue, thick, xshift=0, yshift=2.82cm, domain=240:269] plot(\x:2);
\draw[blue, very thick] (-.06,.82) node {$)$};
\draw[blue, very thick] (-.95,1) [rotate=0]  node [rotate=-20] {$($};
\draw[blue, very thick] (-.3,.6) node {$I$}; 
\filldraw (-.62,.91) circle [radius=1pt];
\draw (-1.8,1.9) node {\tiny{$\gamma_n(q_n)$}};
\draw[->] (-1.5,1.75) -- (-.7,1);
\filldraw (-.45,.87) circle [radius=1pt];
\draw (1,2.1) node {\tiny{$\gamma_n(p_n)$}};
\draw[->] (.6,2) -- (-.34,.9);
\draw (.5,-.5) node {$\mathcal R$};
\end{tikzpicture}
\end{center}
\end{minipage}

\noindent Let $K:=\mu(\{b\}\times\icc{a}{c})$ and let 
\bqn
T:=\mu\{g\in(\deH^2)^{(2)}:\,g\cap I\neq\varnothing\}=i(\mu,\delta_\g)\,.
\eqn
Assume by contradiction  that $K>0$ and choose $M\geq1$ with 
\bqn
M\,K>T\,.
\eqn
Then we can find $M$ elements in $(\gamma_n)_{n\geq1}$, that we rename $\eta_1,\dots,\eta_M$,
and that have the properties that  the intervals
\bqn
 I_{\eta_1(a),\eta_1(c)}\,,  I_{\eta_2(a),\eta_2(c)},\dots, I_{\eta_M(a),\eta_M(c)}
\eqn
are consecutive in $\deH^2$, and have pairwise disjoint closures.
Then it follows that 
\bqn
\mu\left(\bigcup_{i=1}^M(\{\eta_i(b)\}\times\icc{\eta_i(a)}{\eta_i(c)})\right)
=\sum_{i=1}^M\mu(\{\eta_i(b)\}\times\icc{\eta_i(a)}{\eta_i(c)})
=M\, K>T\,,
\eqn
while on the other hand we clearly have
\bqn
\bigcup_{i=1}^M(\{\eta_i(b)\}\times\icc{\eta_i(a)}{\eta_i(c)})\subset\{g\in(\deH^2)^{(2)}:\,g\cap I\neq\varnothing\}\,.
\eqn
This leads to a contradiction and proves the lemma.
\end{proof}
As a corollary we obtain the following:
\begin{cor}\label{lem:1.15}  Let $x_0,\dots,x_n$ be a sequence of consecutive vertices 
of a complementary region $\mathcal R$
labelled in such a way that $(x_n,\dots,x_0)$ is positively oriented.
Then $(x_0,x_n)$ is $\mu$-somewhat short.
\end{cor}

\begin{proof} The proof proceeds by recurrence.
For $n=1$ the statement holds.
Let us now suppose $n\geq2$.  We have the following equalities
\bqn
\ba
\ioo{x_n}{x_0}&=\ioc{x_n}{x_{n-1}}\cup\ioo{x_{n-1}}{x_0}\\
\ioo{x_0}{x_n}&=\ioo{x_{n-1}}{x_n}\cap\ioo{x_{n-2}}{x_n}\cap\ioo{x_0}{x_{n-1}}\,.
\ea
\eqn
Thus 
\bqn
\ba
\ioo{x_n}{x_0}\times\ioo{x_0}{x_n}
\subset&(\ioo{x_n}{x_{n-1}}\times\ioo{x_{n-1}}{x_n})\\
    &\cup\{x_{n-1}\}\times\ioo{x_{n-2}}{x_n}\\
    &\cup\ioo{x_{n-1}}{x_0}\times\ioo{x_0}{x_{n-1}}\,.
\ea
\eqn
Using that $(x_{n-1},x_n)$ is $\mu$-somewhat short, the induction hypothesis that 
$(x_0,x_{n-1})$ is $\mu$-somewhat short and Lemma~\ref{lem:1.14} we get
to the conclusion that $(x_0,x_{n-1})$ is $\mu$-somewhat short.
\end{proof}

We can now complete the proof of the following result, announced at the beginning of the section:
\begin{prop}\label{thm:1.13}  Let $\mu$ be a geodesic current, $\wt \Lambda$ be a $\G$-invariant  lamination consisting of $\mu$-somewhat short geodesics without isolated leaves. Then any geodesic corresponding to a crown is $\mu$-somewhat short.
%
\end{prop}

\begin{proof}

Let $\Cc\subset \Sigma$ be a crown in the complement of the lamination $\wt \Lambda$, and let $\g\in\G$ be a geodesic bounding $\Cc$.  We chose lifts to $\H^2$ in such a way that the half plane to the left of $(\gamma_+,\gamma_-)$ contains a lift $\widetilde\Cc$ of the crown $\Cc$.


\medskip
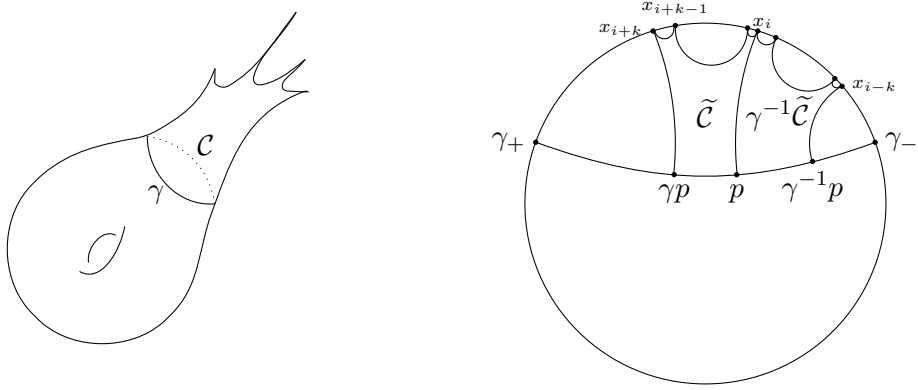
\begin{minipage}{.5\textwidth}
\begin{tikzpicture}[scale=.6]
\draw [rounded corners] 
	(1,2.5) to [out=250, in=30] (-.5,1) 	
	            to [out=190, in=40] (-3,0) 
	            to [out=230, in=130] (-3,-3)  
	            to [out=320, in=220]  (0,-3)
	            to [out=50, in= 250](1,-.5) 
	            to [out=70, in=200] (3,2);
\draw (-.5,1) to [out=270, in =185] (1,-.5);
\draw[dotted] (1,-.5) to [out=100, in=350] (-.5,1);
\draw (-2,-2) to [out=330, in= 260] (-1,-1);
\draw (-1.8,-1.8) to [out=100, in=150] (-1.2,-1.2);
\draw[rounded corners] (1,2.5) to [out=270, in=120] (1.1,2) to [out=30, in=233] (2.8,3.75);
\draw[rounded corners] (2.8,3.75) to [out=239, in=75] (1.8,2.1) to [out=10, in=220] (3, 2.7) ;
\draw[rounded corners] (3, 2.7) to [out=230, in=60] (2.4,2) to [out=360, in=350] (3,2);
\draw (.8,.8) node {$\mathcal C$};
\draw (-.3,-.3) node {$\gamma$};
\end{tikzpicture}
\end{minipage}
\begin{minipage}{.5\textwidth}
\vskip.2cm
\begin{tikzpicture}[scale=.8]
\draw (0,0) circle [radius=3];
\filldraw (-2.82,1.02) circle [radius=1pt] node[left] {$\gamma_+$};
\filldraw (2.82,1.02) circle [radius=1pt] node[right] {$\gamma_-$};
\draw (-2.82,1.02) to [out=340, in=200] (2.83,1.02);
\draw[xshift=-5.5cm, yshift=1cm, domain=354:382] plot(\x:5);
\draw[xshift=5.5cm, yshift=1cm, domain=158:186] plot(\x:5);
\filldraw (-.87,2.87) circle [radius=1pt] node[left] {\tiny{$x_{i+k}$}};
\filldraw (-.5,2.96) circle [radius=1pt] node[above] {\tiny{$x_{i+k-1}$}};;
\filldraw (.7,2.92) circle [radius=1pt];
\filldraw (.87,2.87) circle [radius=1pt];
\draw (.95, 3)  node {\tiny{$x_{i}$}};
\draw [xshift=-.685cm, yshift=2.9cm, domain=190:370] plot(\x:.16);
\draw [xshift=.1cm, yshift=2.9cm, domain=175:360] plot(\x:.6);
\draw [xshift=.778cm, yshift=2.85cm, domain=150:350] plot(\x:.075);
\draw [xshift=1.02cm, yshift=2.8cm, domain=163:343] plot(\x:.16);
\filldraw (1.17,2.76) circle [radius=1pt];
\draw [xshift=1.71cm, yshift=2.47cm, domain=150:320] plot(\x:.6);
\filldraw (2.15,2.08) circle [radius=1pt];
\draw [xshift=2.18cm, yshift=2cm, domain=100:350] plot(\x:.075);
\filldraw (2.27,1.95) circle [radius=1pt] node[right] {\tiny{$x_{i-k}$}};
\draw [xshift=2.83cm, yshift=1.01cm, domain=120:198] plot(\x:1.1);
\filldraw (1.78,.7) circle [radius=1pt] node[below] {$\gamma^{-1} p$};
\filldraw (.52,.48) circle [radius=1pt] node[below] {$p$};
\filldraw (-.52,.48) circle [radius=1pt] node[below] {$\gamma p$};
\draw (0,1.5) node {$\widetilde{\mathcal C}$};
\draw (1.2,1.5) node {$\gamma^{-1}\widetilde{\mathcal C}$};
\end{tikzpicture}
\end{minipage}

\medskip
\noindent
Then $\widetilde\Cc$ has consecutive ideal sides $(x_i,x_{i+1})$, $i\in\Z$, 
labelled in such a way that $(x_i,x_{i+1},x_{i+2})$ is positively oriented. 
Now observe that 
\bqn
(\gamma_-,\gamma_+)=\lim_{n\to\infty}(x_{-n},x_n)\,.
\eqn
By Corollary~\ref{lem:1.15} $(x_{-n},x_n)$ is $\mu$-somewhat short, so that Lemma~\ref{lem:convergence of ss} implies
that $(\gamma_-,\gamma_+)$ is $\mu$-somewhat short, which concludes the proof.
\end{proof}

We can fully understand the support of a basic geodesic current $\mu$ for which there exists a lamination $\wt \L$ of $\mu$-somewhat short geodesic without isolated leaves:

\begin{prop}\label{lem:1.17}  Let $\mu$ be a basic geodesic current, and assume that there exists a lamination $\wt \Lambda$ consisting of $\mu$-somewhat short geodesics without isolated leaves. Then the support of the geodesic current $\mu$ coincides with $\wt\Lambda$.
\end{prop}

\begin{proof} Observe first that the lamination $\wt\Lambda$ fills the supporting subsurface $\Sigma_\mu$ of $\mu$: indeed $\Lambda$ is contained in $\Sigma_\mu$, and the totally geodesic subsurface filled by $\Lambda$ is bounded by the closed geodesics bounding the crowns, that are somewhat short. Since, by definition of a basic geodesic current,  the only somewhat short closed geodesic contained in $\Sigma_\mu$ are its boundary components, we get our first claims. In particular this implies that $\Lambda$ is minimal.

Since $\mu$ is basic, if a geodesic $g$ belongs to $\supp(\mu)$, then $\p_\Sigma(g)$ is contained in the interior of the supporting surface $\Sigma_\mu$. First we show that $\supp(\mu)\subset\wt\Lambda$.  
By contradiction let us assume that there is $g\in\supp(\mu)$, but $g\notin\wt\Lambda$.
Observe that since $g\in\supp(\mu)$, it cannot intersect transversally a $\mu$-somewhat short geodesic.
Together with $g\notin\wt\Lambda$, this implies that $g$ is contained in a complementary region $\Rr$.
As $g$ cannot be a ``side'' of $\Rr$, the region $\Rr$ has at least four vertices 
and $g$ is either a diagonal of $\Rr$ or connects a vertex of $\Rr$ to a cusp or to an endpoint of a boundary geodesic $\gamma$ of $\Sigma_\mu$.
In any case we can find a diagonal, say $\delta$, of $\Rr$ intersecting $g$ transversally.
By Corollary~\ref{lem:1.15} $\delta$ is $\mu$-somewhat short, which leads to a contradiction.
Thus $\supp(\mu)\subset\wt\Lambda$ is a $\Gamma$-invariant sublamination of $\wt\Lambda$ and 
hence it coincides with $\wt\Lambda$ by minimality of $\Lambda$.
\end{proof}

We finish this subsection by establishing two lemmas that will imply the continuity statement alluded to 
after Theorem~\ref{thm:main_currents}.

\begin{lem}\label{lem:1.18}  Let $\Lambda\subset\Sigma$ be a minimal geodesic lamination that is not a closed
geodesic and let $\wt\Lambda\subset\H^2$ be its lift to $\H^2$.
\be
\item For every $(a,b)\in\wt\Lambda$ there exists a sequence $(a_n,b_n)\in\wt\Lambda$ with 
$\lim_n(a_n,b_n)=(a,b)$, $a_n\neq a$ and $b_n\neq b$ for every $n\neq1$.
\item For every positively oriented $(a,c,d)\in(\deH^2)^3$,
the intersection $\{a\}\times\ioo{c}{d}\cap\wt\Lambda$ is either empty or it consists
of one leaf or of two leaves that are the consecutive sides of a complementary region of $\wt\Lambda$.
\ee
\end{lem}

\begin{proof}  (1) Pick a leaf $\lambda=(e,f)$ in $\wt\Lambda$ such that
neither $e$ nor $f$ is $\Gamma$-equivalent to $a$ or $b$.
Since the projection of $\lambda$ is dense in $\Lambda$, there exists a sequence
$\gamma_n\in\Gamma$ with 
\bqn
\lim_n\gamma_n(e,f)=(a,b)\,.
\eqn
Now set $a_n=\gamma_ne$ and $\b_n=\gamma_nf$.

\medskip
\noindent
(2) Assume that $\{a\}\times\ioo{c}{d}\cap\wt\Lambda$ contains three leaves, say
$\lambda_i=(a,c_i)$, for $i=1,2,3$ with $c_2\in\ioo{c_1}{c_3}$.
Now pick a sequence $(a_n,b_n)$ as in (1) with $\lim_n(a_n,b_n)=(a,c_2)$.
For $n$ large enough we will have $b_n\in\ioo{c_1}{c_3}$ and $a_n\in\ioo{c_3}{c_1}$ while $a_n\neq a$ for all $n\geq1$.
But this implies for such an $n$ that $(a_n,b_n)$ crosses either $(a,c_1)$ or $(a,c_3)$,
that is a contradiction.
\end{proof}

\begin{lem}\label{lem:1.19} Let $\mu$ be a basic geodesic current whose support is 
a lamination $\wt\Lambda$. 
Let $a\in\deH^2$ and $I\subset\deH^2$ be an interval such that $a\notin\overline{I}$.  
Then 
\bqn
\mu(\{a\}\times I)=0\,.
\eqn
\end{lem}

\begin{proof}  We start out observing that since $\Lambda$ is minimal and not a closed geodesic,
then $\mu(\{\lambda\})=0$ for every $\lambda\in\wt\Lambda$.
Next, let $c,d\in\deH^2$ with $I\subset\ioo{c}{d}$ and $a\notin\icc{c}{d}$.
Then 
\bqn
\mu(\{a\}\times I)\leq\mu(\{a\}\times\ioo{c}{d})\,.
\eqn
Now apply the preceding lemma. Since $\{a\}\times\ioo{c}{d}\cap\wt\Lambda$ consists of at most two leaves, and the support of the geodesic current $\mu$ coincides with $\wt \Lambda$, we get the desired result.
%
\end{proof}

We deduce immediately the following:

\begin{cor}\label{cor:1.20} Under the hypothesis of Lemma \ref{lem:1.19} the map
\bqn
(a,b,c,d)\mapsto\mu(\icc{a}{b}\times\icc{c}{d})
\eqn
is continuous on the set of positive $4$-tuples of points in $\deH^2$.
\end{cor}

\subsection{Inequalities on intersections, systoles}\label{subsec:1.4}
Let, as usual, $\Sigma=\Gamma\backslash\H^2$ be a complete finite area hyperbolic surface,
$\Gamma<\PSL(2,\bR)$.  

In order to conclude the proof of Theorem \ref{thm:main_currents} we will need a good understanding of the behaviour of the intersection function $i(\mu,c)$
under certain surgeries consisting in reducing the self-intersection number of a closed geodesic $c$. Let $p\in c$ be a self-intersection point of the geodesic $c$.

Let $\gamma\in\pi_1(\Sigma,p)$ be a representative of $c$.  Let us parametrize $\gamma$ by $g:S^1\to\Sigma$.
Then $g^{-1}(p)=\{t_1,\dots, t_\ell\}$ with $\ell\geq2$, where we assume that $t_1,\dots,t_\ell$ are positively oriented in $S^1$.
Define $\gamma_2=g|_{\icc{t_1}{t_2}}$ and $\gamma_3=g|_{\icc{t_2}{t_1}}$; let $\gamma_1:=\gamma^{-1}_2\gamma_3$.

\medskip
\begin{center}
\begin{figure}[h]\label{fig:1}
\begin{tikzpicture}
\draw (0,0) .. controls (1.5,1) and (3,1) .. (3,0);
\draw (3,0) .. controls (3,-1) and (1.5,-1) .. (0,0);
\draw (0,0) .. controls (-1.5,1) and (-3,1) .. (-3,0);
\draw (-3,0) .. controls (-3,-1) and (-1.5,-1) .. (0,0);
\filldraw (0,0) circle [radius=1pt];
\draw (2.2,0) node {\small{$\gamma_2$}};
\draw (-2.2,0) node {\small{$\gamma_3$}};
%
\draw[rounded corners, scale=1.3] (-3,0) .. controls (-3,1) and (-1.5,1) .. (0,0.2) .. controls (1.5,1) and (3,1) .. (3,0);
\draw[rounded corners, scale=1.3] (-3,0) .. controls (-3,-1) and (-1.5,-1) .. (0,-0.2) .. controls (1.5,-1) and (3,-1) .. (3,0);
\draw (2.5,-1.3) node {{$\gamma_1=\gamma_2^{-1}\gamma_3$}};
\draw (0,.9) node {{$\gamma=\gamma_2\gamma_3$}};
\draw (3,0) node [rotate=270] {\tiny{$<$}};
\draw (-3,0) node [rotate=270] {\tiny{$<$}};
\draw (3.9,0) node [rotate=90] {\tiny{$<$}};
\draw (-3.9,0) node [rotate=270] {\tiny{$<$}};
\end{tikzpicture}
\caption{The three curves obtained from $c$ resolving the self intersection at $p$.}
\end{figure}
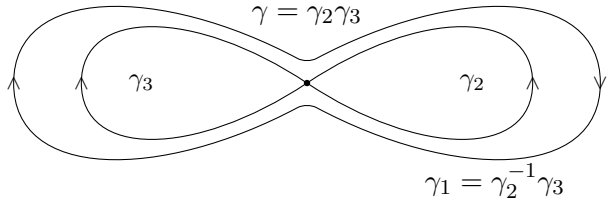
\end{center}

We have:

\begin{prop}\label{prop:1.23}  If $\Sigma$ is not the thrice punctured sphere,
then one of $\gamma_1,\gamma_2$ or $\gamma_3$ must be hyperbolic.  
For such a $\gamma_i$ we have:
\be
\item $i(\mu,\delta_{\gamma_i})\leq i(\mu,\delta_{\gamma})$;
\item The self-intersection number of any $\gamma_i$ is strictly smaller than the one of $\gamma$.
\ee
\end{prop}

For the proof see \cite[Proposition~4.6]{Martone_Zhang}.  We use this to draw the following
conclusion concerning the systole of $\mu$.

\begin{cor}\label{cor:1.24} If $\Sigma$ is not the thrice punctured sphere, then 
\bqn
\Syst(\mu)=\inf\{i(\mu,\delta_\gamma):\,\gamma\in\Gamma,\,\gamma\text{ is hyperbolic and simple}\}\,.
\eqn
\end{cor}

\begin{proof}  Indeed, given any $\gamma$ hyperbolic with self-intersection,
by iterating Proposition~\ref{prop:1.23} we can find $\eta$ hyperbolic and simple with 
$i(\mu,\delta_\eta)\leq i(\mu,\delta_\gamma)$.
\end{proof}

From this we deduce:

\begin{prop}\label{prop:1.25} Let $\mu$ be a basic geodesic current. Assume that the supporting surface $\Sigma_\mu$ is not the thrice-punctured sphere and $\Syst_{\Sigma_\mu}(\mu)=0$.
Then there exists $(a,b)\in(\deH^2)^{(2)}$ which is simple, $\mu$-somewhat short and recurrent.
\end{prop}

\begin{proof}  By Corollary~\ref{cor:1.24}, there is a sequence $(\gamma_k)_{k\geq1}$ 
of hyperbolic elements in $\Gamma$ representing simple closed geodesics and such that 
\bqn
\lim_ki(\mu,\delta_{\gamma_k})=0\,.
\eqn
Since $i(\mu,\delta_{\gamma_k})>0$ for every $k>1$, we must have for the sequence
of hyperbolic lengths $\ell(\gamma_k)$, that $\lim_k\ell(\gamma_k)=\infty$.
Passing to a subsequence, and possibly conjugating the elements $\gamma_k$ so that their axis intersect a fixed fundamental region, we may assume that
\bqn
\lim_k(\gamma_{k,-},\gamma_{k,+})=(a,b)\,,
\eqn
where $(a,b)$ is simple as well, non-isolated, and hence neither $a$ nor $b$ can be cusps or fixed points of boundary elements.\\
\begin{minipage}{.5\textwidth}
We now show that $(a,b)$ is $\mu$-somewhat short.  
Let $(x,a,y,z,b,w)$ be positive.  
Let $V_a$ and $V_b$ be closed neighbourhoods respectively of $a$ and $b$,
such that 
\bqn
\{x,y,z,w\}\cap(V_a\cup V_b)=\varnothing\,.
\eqn
For $k$ large enough we have that $\gamma_{k,-}\in V_a$ and $\gamma_{k,+}\in V_b$.
\end{minipage}
\begin{minipage}{.5\textwidth}
\medskip
\begin{center}
\begin{tikzpicture}[scale=.8]
\draw (0,0) circle [radius=2];
\filldraw (-1,1.732) circle [radius=1pt] node[above left] {$x$};
\filldraw (-1.732,-1) circle [radius=1pt] node[left] {$y$};
\filldraw (1.732,-1) circle [radius=1pt] node[right] {$z$};
\filldraw (0,2) circle [radius=1pt] node[above] {$w$};
\draw[blue, very thick] (-1.732,1) arc[start angle=150, end angle=200, radius=2];
\draw[blue] (-1.7,0) node {$V_a$};
\draw[blue, very thick] (1,1.732) arc[start angle=420, end angle=350, radius=2];
\draw[blue] (1.5,.6) node {$V_b$};
\filldraw[blue] (-1.88,-.68) circle [radius=1pt];
\filldraw[blue] (-1.732,1) circle [radius=1pt];
\filldraw[blue] (1,1.732) circle [radius=1pt];
\filldraw[blue] (1.97, -.35) circle [radius=1pt];
\filldraw (-1.9,.62) circle [radius=1pt] node[left] {$a$};
\filldraw (-2,0) circle [radius=1pt] node[left] {$\gamma_{k,-}$};
\filldraw (2,0) circle [radius=1pt] node[right] {$b$};
\filldraw (1.732,1) circle [radius=1pt] node[above right] {$\gamma_{k,+}$};
\end{tikzpicture}
\end{center}
\end{minipage}

\medskip

Since $\lim_k\ell(\gamma_k)=\infty$, this implies that 
\bqn
\gamma_ky\in\ioo{z}{w}
\eqn
for $k$ large enough, and hence, again for large $k$,
\bqn
      \mu(\ioo{w}{x}\times\ioo{y}{z})
\leq\mu(\ioo{\gamma_{k,+}}{\gamma_{k,-}}\times\ioc{y}{\gamma_ky})
   =i(\mu,\delta_{\gamma_k})\,.
\eqn
This implies that $\mu(\ioo{w}{x}\times\ioo{y}{z})=0$,
and hence $(a,b)$ is $\mu$-somewhat short.
\end{proof}
\subsection{The thrice punctured sphere case}\label{subsec:thrice}
The first step in the proof of Theorem \ref{thm:main_currents} consists of the specific case of subsurfaces isomorphic to pairs of pants. It is well known that such subsurfaces do not support any non-trivial compactly supported measured lamination. However there are many interesting geodesic currents on pairs of pants: for example the theory of maximal and Hitchin representations on a pair of pants is quite rich
. The aim of this section is to prove the following restatement of Theorem \ref{thm:main_currents} in the case of 3-punctured spheres:

\begin{prop}\label{prop_intro:prop1}
Let $\Sigma=\Sigma_{0,3}$ be a 3-punctured sphere. For every non-vanishing geodesic current $\mu$ on $\Sigma$, ${\rm Syst}(\mu)$ is positive.
\end{prop}

In the whole section we let $\Gamma$ be the fundamental group of the 3-punctured sphere, $\Gamma:=\pi_1(\Sigma_{0,3})$,
with presentation
\bqn
\Gamma=\langle a,b,c:\,abc=e\rangle\,.
\eqn
An element $\g\in\Gamma$ is \emph{peripheral} if it is conjugated to either $a^k$, $b^k$ or $c^k$ for some $k\in\Z$, $k\neq0$.

\begin{lem}\label{lem:a}
Let $\g_1,\g_2\in\Gamma$. Assume that $\g_1\g_2$ is not peripheral, but $\g_1,\g_2$ and $\g_1\g_2^{-1}$ are. 
Then $\gamma_1$, $\gamma_2$ and $\gamma_1\gamma_2^{-1}$ represent distinct primitive boundary components
and $\g_1\g_2$ is conjugated to one of the following: $(ab^{-1})^{\pm1}$, $(bc^{-1})^{\pm1}$, $(ca^{-1})^{\pm1}$.
\end{lem}

\begin{proof}
Without loss of generality we can assume that $\g_1=a^k$ and that,
applying if necessary an automorphism fixing $a$ and exchanging $b$ and $c$, 
either we have $\g_2=wb^lw^{-1}$ or $\g_2=wa^lw^{-1}$ for some word $w$ in $a$ and $b$. 

We focus on the first case first. Since $\g_1\g_2^{-1}$ is peripheral, and $k$  and $l$ are non zero, 
we deduce, projecting to the abelianization $\Z^2$ of $\Gamma$, that $k=-l$ and $\g_1\g_2^{-1}=a^kwb^{k}w^{-1}$ is conjugated to $(ab)^k$. 
This implies that the cyclically reduced expression for $\g_1\g_2^{-1}$ is $(ab)^k$. 
We write $w=a^sw_0b^t$ where $w_0$ is either trivial or begins with a power of $b$ and ends with a power of $a$. 
The expression  $\g_1\g_2^{-1}=a^kw_0b^{k}w_0^{-1}$ is cyclically reduced. 
Hence $k=\pm1$ which means that $\g_1\g_2$ is conjugated to $(ab^{-1})^{\pm1}$.

We conclude the proof by showing that the second case cannot happen. 
We argue by contradiction.  
Looking, once more, at the abelianization, we would deduce that $\g_1\g_2^{-1}$ is conjugated to $a^{k-l}$. 
The word $w$ cannot be a power of $a$, since otherwise $\g_1\g_2=a^{k+l}$ would be peripheral. 
Hence we can write $w$ as $a^mw_0 a^n$ with $w_0$ beginning and ending with a power of $b$. 
However this implies that $\g_1\g_2^{-1}=a^kw_0a^{-l}w_0^{-1}$. 
This is absurd since such an expression is cyclically reduced, 
but we know that $\g_1\g_2^{-1}$ is conjugated to a power of $a$.
\end{proof}

\begin{proof}[Proof of Proposition~\ref{prop_intro:prop1}]
Assume that ${\rm Syst}(\mu)=0$. 
If there exists $\gamma\in\Gamma$ non peripheral with $i(\mu,\delta_\g)=0$, 
it follows from Proposition \ref{prop:L(rho(gamma))=0} that for every $\g\in\G$, $i(\mu,\delta_\gamma)=0$ and we are done. 

Let us then assume by contradiction that no non-peripheral element has $i(\mu,\delta_\g)=0$, 
but there is a sequence $\{\g_n\}_{n\in\N}$ of non-peripheral elements with $i(\mu,\g_n)<1/n$.
Combining Lemma~\ref{lem:a} and Proposition~\ref{prop:1.23} we can assume that $\g_n$ has a single self-intersection: 
since $\g_n$ is by assumption non-peripheral, $\g_n$ cannot be simple, if it has more than one self intersection points, 
we can choose one such a self-intersection point $p$ and deduce from Lemma~\ref{lem:a} 
that at least one of the three loops obtained resolving the intersection at $p$ is non-peripheral. 
Moreover it follows from Proposition~\ref{prop:1.23} that the intersection of such a loop with the geodesic current $\mu$ is not bigger than the one of $\g_n$.

The contradiction follows from the fact that every loop in $\Sigma_{0,3}$ with a single self-intersection is conjugated 
to one of $(ab^{-1})^{\pm1}$, $(bc^{-1})^{\pm1}$, or $(ca^{-1})^{\pm1}$, and therefore the sequence is eventually constant. 
\end{proof}

\subsection{Completion of the proof of Theorem~\ref{thm:main_currents}}\label{subsec:1.5}
As a consequence of Proposition \ref{prop_intro:prop1} we can and will assume that $\Sigma_\mu$ is not a trice punctured sphere.

\eqref{item:main_currents(1)}$\Rightarrow$\eqref{item:main_currents(2)}.  
This is a consequence of Proposition~\ref{prop:1.25}.

\medskip
\noindent
\eqref{item:main_currents(2)}$\Rightarrow$\eqref{item:main_currents(3)}.
Let us assume that $(a,b)$ is $\mu$-somewhat short and  recurrent.
We can apply Proposition~\ref{lem:simple_ss_geod} and find $(a',b')$ $\mu$-somewhat short,  simple and recurrent.  The closure of $\p_\Sigma((a',b'))$ in $\Sigma$ contains
then a maximal sublamination $\wt\Lambda$ that is not a closed geodesic.  Finally we deduce from
Proposition~\ref{thm:1.13} that $\wt \Lambda$ is surface filling and from Proposition~\ref{lem:1.17} that the support of $\mu$ coincides with $\wt \Lambda$.

\medskip
\noindent
\eqref{item:main_currents(3)}$\Rightarrow$\eqref{item:main_currents(1)}.  
Let $\Lambda$ be this minimal lamination and $m$ the transverse invariant measure defined by $\mu$.
Since the lamination is minimal and not reduced to a closed geodesic, there is, for each $\epsilon>0$,
a geodesic segment $\sigma_\epsilon\subset\Sigma$ transverse to $\Lambda$ 
with $\lambda(\sigma_\epsilon)<\epsilon$.  
Given $x\in\sigma_\epsilon$, the concatenation of the subsegment of $\Lambda$ between $x$ 
and the first return to $\sigma_\epsilon$ gives a closed loop
and an element $\gamma\in\Gamma$ with $i(\mu,\delta_\gamma)<\epsilon$.

\section{Currents with positive systole}\label{sec:pos}
The main objective of this section is to prove Theorem \ref{thm_intro:posSyst} in the introduction. This rests on the following characterization of currents with vanishing systole. Recall that $\MLc (\Sigma)$ is the space of currents with compact carrier and vanishing self intersection, and $\CcK(\Sigma)$ is the space of currents with carrier in a compact subset $K\subset \Sigma$.
\begin{thm}\label{thm:intersection}
Let $K\subset \Sigma$ be a compact subset such that $\MLc (\Sigma)\subset\CcK(\Sigma)$. For a current $\mu\in\mathcal C(\Sigma)$ the following are equivalent
\begin{enumerate}
\item $\Syst(\mu)=0$;
\item the function $\lambda\mapsto i(\mu,\lambda)$ admits a zero on $\MLc(\Sigma)\setminus\{0\}$;
\item the function $\lambda\mapsto i(\mu,\lambda)$ admits a zero on $\CcK(\Sigma)\setminus\{0\}$;
\end{enumerate}
\end{thm}
\begin{proof}
We may assume that $\mu\neq 0$.

\medskip
\noindent
(1)$\Rightarrow$(2). If $i(\mu,\delta_\g)=0$ for some $\g\in\G$ hyperbolic, then (see Corollary \ref{cor:1.24}) there is $\eta\in\Gamma$ representing a simple closed geodesic with $i(\mu,\delta_\eta)=0$ and (2) holds. If, instead, $i(\mu,\delta_\g)>0$ for every hyperbolic $\g\in\G$, then it follows from Theorem \ref{thm:main_currents} that $i(\mu,\mu)=0$, which, again, shows (2).
 
 \medskip
\noindent
(2)$\Rightarrow$(3). This is clear by the choice of $K$.

\medskip
\noindent
(3)$\Rightarrow$(1). Let $\l_0\in\CcK(\Sigma)\setminus\{0\}$ with $i(\mu,\l_0)=0$, and fix $(a,b)\in\supp(\l_0)$. 
Since $\p_\Sigma((a,b)) \subset K $ this implies in particular that the geodesic $(a,b)$ is recurrent. 
Let $x,y,z,t\in\partial\H^2$ be such that $(x,a,y,z,b,t)$ is positively oriented. 
Then $\ioo{x}{y}\times \ioo{z}{t}$ is a neighbourhood of $(a,b)$ hence $\l_0(\ioo{x}{y}\times \ioo{z}{t})>0$. 
But then it follows from $i(\mu,\l_0)=0$ that $\mu(\ioo{y}{z}\times \ioo{t}{x})=0$. 
This implies that $\mu(\ioo{a}{b}\times \ioo{b}{a})=0$ hence $(a,b)$ is a $\mu$-somewhat short geodesic which implies $\Syst(\mu)=0$ by Theorem \ref{thm:main_currents}.
\end{proof}

Let now $\l_0\in\CcK(\Sigma)$ be a surface filling current with compact support. Then 
$$\mathcal C(\Sigma)_{\l_0}:=\{\mu\in\mathcal C(\Sigma)|\; i(\mu,\l_0)=1\}$$
is compact and so is 
$$\MLc(\Sigma)_{\l_0}=\MLc(\Sigma)\cap \mathcal C(\Sigma)_{\lambda_0}.$$
\begin{cor}\label{cor:open}
The subset 
$$\Syst(0)=\{\mu\in\mathcal C(\Sigma)|\;\Syst(\mu)=0\}$$
is closed.
\end{cor}
\begin{proof}
Assume that $\mu_n\in\Syst(0)$ is a convergent sequence with limit $\mu\in \mathcal C(\Sigma)$. By Theorem \ref{thm:intersection} there is $\lambda_n\in\MLc(\Sigma)_{>0}$ with $i(\mu_n,\l_n)=0$. Passing to a subsequence we may assume that $\l_n$ converges to $\l\in\MLc(\Sigma)_{\l_0}$. By the continuity of the intersection this implies that $i(\mu,\lambda)=\lim i(\mu_n,\l_n)=0$ and since $\l\neq 0$, Theorem \ref{thm:intersection} implies $\Syst(\mu)=0$.
\end{proof}
Let $\P\Syst(0)$ denote the image of $\Syst(0)\setminus\{0\}$ in $\P(\mathcal C (\Sigma))$ and let $\Omega$ be its complement, which is open by Corollary~\ref{cor:open}. When denoting elements of $\P(\mathcal C(\Sigma))$ we will always assume that $\nu\in\mathcal C(\Sigma)_{\l_0}$.
\begin{cor}
Let $K\subset\Sigma$ be compact with $\MLc(\Sigma)\subset\CcK(\Sigma)$ and $[\mu]\in\Omega$. Then there is an open neighbourhood $V_{[\mu]}$ of $[\mu]$ and constants constants $0<c_1\leq c_2$ such that for every $[\nu]\in V_{[\mu]}$ and every geodesic $c$ contained in $K$ we have 
$$c_1\ell(c)\leq i(\nu,\delta_c)\leq c_2\ell(c).$$
\end{cor}
\begin{proof}
We prove the lower bound; the proof of the upper bound is completely analogous. Assume by contradiction that there exists a sequence $\nu_n$ converging to $\mu$ and a sequence of $\g_n\in\G$ with $\delta_{\g_n}\in\CcK(\Sigma)$ with 
$$\lim\frac{i(\nu_n,\delta_{\gamma_n})}{\ell(\g_n)}=0.$$
Recall that if $\mathcal L$ denotes the Liouville current, 
$$\CcK(\Sigma)_\mathcal L:=\{\lambda\in\CcK(\Sigma)|\; i(\mathcal L, \lambda)=1\}$$
is compact. Since $\delta_{\g_n}/\ell(\g_n)\in\CcK(\Sigma)_\mathcal L$, we may assume, passing to a subsequence, that  $\delta_{\g_n}/\ell(\g_n)$ converges to $\l\in \CcK(\Sigma)_\mathcal L$. But then $i(\mu,\lambda)=0$ and hence $\Syst(\mu)=0$ by Theorem \ref{thm:intersection}. This is a contradiction.
\end{proof}

\section{An application to the length spectrum
  compactification}\label{sec:5}

%
Consider the map $\calC(\Sigma)\to\R_{\geq0}^\mathscr{C}$ 
that to a current $\mu$ associates the function $c\mapsto i(\mu,c)$.
It descends to a well defined continuous map
\bqn
I:\P\calC(\Sigma)\longrightarrow\P(\R_{\geq0}^\mathscr{C})\,.
\eqn
It follows from \cite{Martone_Zhang} that both $\calL(\mathrm{Max}(\Sigma,n))$ and $\calL(\mathrm{Hit}(\Sigma,n))$
are in the image of $I$, that is $\calL(\Sigma,n):=\calL(\calX(\Sigma,n))$ is in the image of $I$.
Since $I$ is continuous and $\P\calC(\Sigma)$ is compact, 
we deduce that $\overline{\calL(\Sigma,n)}$ is in the image of $I$.
In particular for $[L]\in\partial\calL(\Sigma,n)$ 
there is a geodesic current $\mu\in\calC(\Sigma)$ such that $L(c)=i(\mu,c)$ for every $c\in\calC$.
Then Theorem~\ref{thm_intro:dec} applied to $\mu$ implies immediately Corollary~\ref{cor_intro:1.5}.

To show Corollary~\ref{cor_intro:1.6} we use Otal's result that $I$ is a homeomorphism onto its image, \cite{Otal}.
As a result we have that 
$$\Omega(\Sigma,n)=I(\Omega)\cap\partial\calL(\Sigma,n)\,,$$
where $\Omega$ is as in Theorem~\ref{thm_intro:posSyst}.  
Thus $\Omega(\Sigma,n)$ is open in $\partial\calL(\Sigma,n)$
and the statements about the length functions in $\Omega(\Sigma,n)$ follow
from the corresponding ones in Theorem~\ref{thm_intro:posSyst}.



\end{document}